\documentclass[aihp]{imsart}

\RequirePackage{amsthm,amsmath,amsfonts,amssymb}
\RequirePackage[numbers]{natbib}
\RequirePackage[colorlinks,citecolor=blue,urlcolor=blue]{hyperref}
\RequirePackage{graphicx}

\usepackage{amssymb}
\usepackage{amsmath, graphicx, rotating}
\usepackage{color}
\usepackage{soul}
\usepackage[dvipsnames]{xcolor}
   
\usepackage{ifthen}
\usepackage{xkeyval}
\usepackage{todonotes}

\usepackage[T1]{fontenc}
\usepackage{lmodern}
\usepackage[english]{babel}

\usepackage{ upgreek }
\usepackage{stmaryrd}
\SetSymbolFont{stmry}{bold}{U}{stmry}{m}{n}
\usepackage{amsthm}
\usepackage{float}

\usepackage{ bbm }
\usepackage{ stmaryrd }
\usepackage{ mathrsfs }
\usepackage{ frcursive }
\usepackage{ comment }

\usepackage{pgf, tikz}
\usetikzlibrary{shapes}
\usepackage{varioref}
\usepackage{enumitem}

\usepackage{mathtools}

\usepackage{dsfont}

\usepackage{ulem}

\startlocaldefs

\definecolor{rouge}{rgb}{0.7,0.00,0.00}
\definecolor{vert}{rgb}{0.00,0.5,0.00}
\definecolor{bleu}{rgb}{0.00,0.00,0.8}

\newtheorem{theorem}{Theorem}[section]
\newtheorem*{theorem*}{Theorem}
\newtheorem{lemma}[theorem]{Lemma}

\newtheorem{proposition}[theorem]{Proposition}

\labelformat{hypothesis}{\textbf{M\kern-0.1mm#1}}

\newtheorem{condition}{Condition}
\newtheorem{conditionA}{A\kern-0.1mm}
\labelformat{conditionA}{\textbf{A\kern-0.1mm#1}}

\renewcommand\dots{\hbox to 1em{.\hss.\hss.}}

\theoremstyle{definition}
\newtheorem{example}[theorem]{Example}
\newtheorem{remark}[theorem]{Remark}

\numberwithin{equation}{section}

\newcommand*{\norm}[1]{\left\lVert#1\right\rVert}

\newcommand*{\scal}[2]{\left\langle {#1}, {#2} \right\rangle}

\def\bb#1{\mathbb{#1}}

\def\geq{\geqslant}
\def\leq{\leqslant}

\newcommand\ee{\varepsilon}

\DeclareMathOperator{\supp}{supp}

\DeclarePairedDelimiter\floor{\lfloor}{\rfloor}

\endlocaldefs

\begin{document}

\begin{frontmatter}

\title{Limit theorems for first passage times of \\ multivariate perpetuity sequences}
\runtitle{Limit theorems for multivariate perpetuity sequences}

\begin{aug}
%
\author[A]{\inits{S.}\fnms{Sebastian}~\snm{Mentemeier}\ead[label=e1]{mentemeier@uni-hildesheim.de}}
\and
\author[B]{\inits{H.}\fnms{Hui}~\snm{Xiao}\ead[label=e2]{xiaohui@amss.ac.cn}}

\address[A]{Universit\"at Hildesheim, Institut f\"ur Mathematik und Angewandte Informatik, Hildesheim 31141, Germany\printead[presep={,\ }]{e1}}

\address[B]{Academy of Mathematics and Systems Science, Chinese Academy of Sciences, Beijing 100190, China\printead[presep={,\ }]{e2}}
\end{aug}

\begin{abstract}
We study the first passage time $\tau_u = \inf \{ n \geq 1: |V_n| > u \}$ for the multivariate perpetuity sequence $V_n = Q_1 + M_1 Q_2 + \cdots + (M_1 \ldots M_{n-1}) Q_n$, where $(M_n, Q_n)$ is a sequence of independent and identically distributed random variables with 
$M_1$ a $d \times d$ ($d \geq 1$) random matrix with nonnegative entries, and $Q_1$ a nonnegative random vector in $\mathbb R^d$. Here $|\cdot|$ denotes the vector norm. The exact asymptotic for the probability $\mathbb P (\tau_u < \infty)$ as $u \to \infty$ has been found by Kesten (Acta Math. 1973). In this paper we prove a conditioned weak law of large numbers for $\tau_u$:  conditioned on the event $\{ \tau_u < \infty \}$, $\frac{\tau_u}{\log u}$ converges in probability to a certain constant $\rho > 0$ as $u \to \infty$. A conditioned central limit theorem for $\tau_u$ is also obtained. We further establish precise large deviation asymptotics for the lower probability $\mathbb P (\tau_u \leq (\beta - l) \log u)$ as $u \to \infty$, where $\beta \in (0, \rho)$ and $l \geq 0$ is a vanishing perturbation satisfying $l \to 0$ as $u \to \infty$. Our results extend those of Buraczewski et al. (Ann. Probab. 2016) from the univariate case ($d=1$) to the multivariate case ($d>1$).
As consequences, we deduce exact asymptotics for the pointwise probability $\mathbb P (\tau_u = \floor{(\beta - l) \log u}  )$ and the local probability $\mathbb P (\tau_u - (\beta - l) \log u  \in (a, a + m ] )$, where $a<0$ and $m \in \mathbb Z_+$. We also establish analogous results for the first passage time $\tau_u^y = \inf \{ n \geq 1: \langle y, V_n \rangle > u \}$, where $y$ is a nonnegative vector in $\mathbb R^d$ with $|y| = 1$. 
\end{abstract}

\begin{abstract}[language=french]
Nous \'etudions le temps de premier passage $\tau_u = \inf \{ n \geq 1: |V_n| > u \}$ pour la s\'equence de perp\'etuit\'e multivari\'ee $V_n = Q_1 + M_1 Q_2 + \cdots + (M_1 \ldots M_{n-1}) Q_n$, o\`u $(M_n, Q_n)$ est une s\'equence de variables al\'eatoires ind\'ependantes et identiquement distribu\'ees, avec $M_1$ une matrice al\'eatoire $d \times d$ ($d \geq 1$) \`a entr\'ees non n\'egatives, et $Q_1$ un vecteur al\'eatoire non n\'egatif dans $\mathbb R^d$. Ici, $|\cdot|$ d\'esigne la norme du vecteur. L'asymptotique exacte de la probabilit\'e $\mathbb P (\tau_u < \infty)$ lorsque $u \to \infty$ a \'et\'e trouv\'ee par Kesten (Acta Math. 1973). Dans cet article, nous prouvons une loi des grands nombres faible conditionn\'ee pour $\tau_u$: conditionnellement \`a l'\'ev\'enement $\{ \tau_u < \infty \}$, $\frac{\tau_u}{\log u}$ converge en probabilit\'e vers une certaine constante $\rho > 0$ lorsque $u \to \infty$. Un th\'eor\`eme limite central conditionn\'e pour $\tau_u$ est \'egalement obtenu. Nous \'etablissons en outre des asymptotiques pr\'ecises de grandes d\'eviations pour la probabilit\'e inf\'erieure $\mathbb P (\tau_u \leq (\beta - l) \log u)$ lorsque $u \to \infty$, o\`u $\beta \in (0, \rho)$ et $l \geq 0$ est une perturbation tendant vers z\'ero lorsque $u \to \infty$. Nos r\'esultats \'etendent ceux de Buraczewski et al. (Ann. Probab. 2016) du cas univari\'e ($d=1$) au cas multivari\'e ($d>1$). Comme cons\'equences, nous d\'eduisons les asymptotiques exactes pour la probabilit\'e ponctuelle $\mathbb P (\tau_u = \floor{(\beta - l) \log u}  )$ et la probabilit\'e locale $\mathbb P (\tau_u - (\beta - l) \log u  \in (a, a + m ] )$, o\`u $a<0$ et $m \in \mathbb Z_+$. Nous \'etablissons \'egalement des r\'esultats analogues pour le temps de premier passage $\tau_u^y = \inf \{ n \geq 1: \langle y, V_n \rangle > u \}$, o\`u $y$ est un vecteur non n\'egatif dans $\mathbb R^d$ avec $|y| = 1$. 
\end{abstract}

\begin{keyword}[class=MSC]
\kwd[Primary ]{60F05}
\kwd{60F10}
\kwd[; secondary ]{60B20}
\kwd{60G70}
\end{keyword}

\begin{keyword}
\kwd{multivariate perpetuity sequences}
\kwd{first passage time}
\kwd{law of large numbers}
\kwd{central limit theorem}
\kwd{large deviations}
\kwd{products of random matrices}
\end{keyword}

\end{frontmatter}

\section{Introduction and main results}

\subsection{Introduction}
For any integer $d \geq 1$, 
equip the Euclidean space $\bb R^d$ with the canonical basis $(e_i)_{1\leq i\leq d}$, 
the scalar product $\langle u, v \rangle = \sum_{i = 1}^d u_i v_i$,
and the vector norm $|v| = \sum_{i = 1}^d \langle v, e_i \rangle$. 
Let $\mathbb{R}^{d}_+ = \{ v \in \bb R^d: \langle v, e_i \rangle \geq 0, \forall 1 \leq i \leq d \}$ be the positive quadrant of $\bb R^d$
and set $\mathbb{S}^{d-1}_{+} = \{v \in \mathbb{R}^{d}_+,  |v| = 1\}$. 
Denote by $\mathcal{M}$ the set of $d\times d$  matrices with non-negative entries, 
equipped with the operator norm $\|M \| = \sup \{ |M v|: v \in \mathbb{S}_+^{d-1} \}$. 
Let $(M_n, Q_n)_{n \geq 1}$ be a sequence of 
independent and identically distributed (i.i.d.) random variables taking values in $\mathcal{M} \times \bb R_+^d$,
defined on some probability space $(\Omega, \mathscr F, \bb P)$.
Note that we do not assume that $M_n$ and $Q_n$ are independent. 

In this paper 
we are interested in finding precise asymptotic properties of the first passage time $\tau_u = \inf \{ n \geq 1: |V_n| > u \}$
of the following multivariate perpetuity sequence generated by positive random matrices: with $V_1 = Q_1$, 
\begin{align}\label{Def_Yn01}
V_n = Q_1 + M_1 Q_2 + \cdots + (M_1 \ldots M_{n-1}) Q_n,  \quad  n \geq 2. 
\end{align}
We will also study the directional first passage time $\tau_u^y = \inf \{ n \geq 1: \langle y, V_n \rangle > u \}$  
for any $y \in \bb S_+^{d-1}$.
From a theoretical perspective, the perpetuity sequence is closely related to 
random walks in random environment \cite{KKS75,Key87}, 
weighted branching processes and branching random walks \cite{Gui90, Liu00, AI09, Bur09, Wan13}.
In financial mathematics, it is particularly useful to analyze  
 the ARCH and GARCH time series models \cite{dRR89, BDM02, Mik04}. 
Recently, it has been used for the analysis of stochastic gradient descent in the theory of machine learning \cite{GSZ21, HM21}. 
 Section \ref{sect:examples} contains various examples where our results can be directly applied.

It is well known that if $\bb E (\log \|M_1\|) + \bb E (\log^+ |Q_1|) < \infty$
and the upper Lyapunov exponent of $(M_n)$ is negative, 
then $V_n$ converges almost surely to the random variable
\begin{align*}
V = Q_1 + \sum_{n = 2}^{\infty} (M_1 \ldots M_{n-1}) Q_n, 
\end{align*} 
which is a distributional fixed point of the equation $V \overset{d}{=} M V + Q$; where $\stackrel{d}{=}$ means same law,  
 $(M, Q)$ is an independent copy of $(M_1, Q_1)$, and $V$ is independent of $(M, Q)$. 
The study of asymptotic properties of $V_n$ was pioneered by 
Kesten \cite{Kes73}, who showed that, 
under suitable conditions on $(M_1, Q_1)$, there exist constants $\alpha > 0$ and $\mathscr C >0$ such that, 
as $u \to \infty$, 
\begin{align}\label{Exit_tail_V_n-intro}
\bb P ( \tau_u < \infty ) = \bb P \left( |V| > u \right)
\sim  \mathscr C  u^{-\alpha},
\end{align}
where $f(u) \sim g(u)$ means $f(u)/g(u) \to 1$ as $u \to \infty$. 
Nowadays there is a lot of work concerned with tail behavior of perpetuity sequences, see in particular Goldie \cite{Gol91}
and the book \cite{BDM16} as well as the references therein. 

Our aim is to provide limit theorems for the (normalized) first passage time of $V_n$, 
conditioned on being finite, 
in the spirit of results for first passage time of random walks, motivated by the following connection. As it will become transparent in the results below, the behavior of first passage times for $|V_n|$ will be similar (in many cases) to that of $e^{S_n}$, with $S_n:=\log \norm{M_1 \cdots M_n}$, the latter being a random walk in the one-dimensional setting, and a subadditive sequence in our multivariate setting.

For fundamental contributions to the study of limit theorems  for first passage times of random walks, 
we refer to von Bahr (\cite{Von 74}, Section 8) and Siegmund (\cite{Sie75}, Theorem 2) for a central limit theorem, 
and to Lalley (\cite{Lal84}, Theorem 5) and Buraczewski and Ma\'slanka \cite{BM19} for a large deviation principle.
In the context of compound poisson processes in collective risk theory, such results were obtained even earlier by Arfwedson \cite{Arf55}
and Segerdahl \cite{Seg55}, 
see also Asmussen \cite[Chapter IV, Section 4]{Asm00}.  
First passage times of more general sequences in $\bb R$ were considered, 
for example, by Grandell \cite{Gra91} and Nyrhinen \cite{Nyr94}, 
while multivariate sequences taking values in $\bb R^d$ were studied by 
Collamore \cite{Col98}.

Conditioned limit theorems for the first passage time $\tau_u$ of perpetuities in the one-dimensional case $d=1$ were studied in Buraczewski et al. \cite{BCDZ16, BDZ18}, 
where the general case of $Q \in \bb R$ was also included.
One of the major findings in those papers is that, while the lower large deviations of perpetuity sequences 
and random walks behave similarly, the upper  large deviations of perpetuity sequences exhibit significantly different 
behavior compared with that of random walks, see \cite[Theorem 2.3]{BCDZ16} and the discussion in \cite{BDZ18}.
Motivated by their results, we are going to prove a conditioned law of large numbers, a conditioned central limit theorem
and precise lower large deviation asymptotics 
for both the first passage time $\tau_u$ and the directional first passage time $\tau_u^y$
for multivariate perpetuity sequences. 
By allowing for a vanishing perturbation in the precise large deviation asymptotics 
(see Theorem \ref{Thm_LD_Perpe_Petrov}), we are able to obtain pointwise estimates for the law of $\tau_u$
without relying on a density assumption on $(M_1, Q_1)$ used in \cite{BDZ18}, see Theorem \ref{Cor_LLT_LD_tau} for details. 
Note that, however, we do not address upper large deviations for the case of negative $Q$, which is currently out of reach.

Closely related	to the perpetuity sequence $V_n$ is the so-called {\it forward process}, defined by
\begin{align}\label{Def_Yn02}
	V_n^* := Q_n + M_n Q_{n-1} + \cdots + (M_n \ldots M_2) Q_1,  \quad  n \geq 1. 
\end{align}
Then the process $\{ V_n^* \}$ satisfies the following random difference equation:
with the convention that $V_0^* = 0$, 
\begin{align}\label{Recur_Equ}
	V_n^* = M_n V_{n-1}^* + Q_n,  \quad  n \geq 1.  
\end{align}
Comparing \eqref{Def_Yn01} and \eqref{Def_Yn02}, it is easy to see that for each $n \ge 1$, the marginal distributions $V_n$ and $V_n^*$ of the perpetuity sequence and the corresponding forward process coincide. 
		This amounts to the equality
	\begin{align} \label{eq:forward process exceedances}
		 \bb P(\tau_u \leq n) = \bb P(|V_n| >u)= \bb P(|V_n^*| >u)
	\end{align}
	which allows to study finite-time exceedances of $V_n^*$ as well, see also the discussion in \cite{BCDZ16}. In \cite{GSZ21}, for example, the process $V_n^*$ is used to describe the distance to the minimum after $n$ iterations of stochastic gradient descent, hence large deviations at time $n$ are of interest. Note however, that in \cite{GSZ21}, $M$ is assumed to be a (symmetric) matrix with real-valued entries. Our setting with nonnegative matrices is relevant for example when studying finite time exceedances of GARCH(p,q) processes, which have not been started in the stationary regime.

We now start by first giving all the necessary assumptions and notations, to be able to state our results in full detail.

\subsection{Assumptions and preliminaries}\label{subsec.notations}

Let $\mu$ denote the distribution of $M_1$. 
Consider the random matrix products
\begin{align*}
\Pi_n = M_1 \ldots M_n. 
\end{align*} 
Set 
$$ 
I_{\mu} = \{ s \geq 0: \bb E(\| M_1 \|^s ) < + \infty\}.
$$
By H\"{o}lder's inequality, it is easy to see that $I_{\mu}$ is an interval of $\bb R$.
Let $I_\mu^\circ$ be the interior of $I_\mu$. 
By \cite{BM16},  for any $s\in I_{\mu}$, the following limit exists: 
\begin{align}\label{def-kappa}
\kappa(s) = \lim_{n\to\infty} \left(\bb E  \| \Pi_n \|^{s}\right)^{\frac{1}{n}}, 
\end{align}
and the function $\Lambda = \log\kappa: I_{\mu} \to \bb R$ is convex and analytic on $I_{\mu}^{\circ}$, 
with $\Lambda(0) = 0$. 
We set $\iota(M) = \inf_{x\in \bb S_+^{d-1}} |Mx|$ and $\kappa_Q(s) = \bb E (|Q_1|^s)$ for $s \geq 0$. 
The constant $\alpha$ is introduced by the following assumption.

\begin{conditionA}\label{Condi_ExpMom}
There exist $\alpha \in I_{\mu}^{\circ}$ and $\eta > 0$ such that 
\begin{align}\label{Def_alpha}
\Lambda(\alpha) = 0,  
\end{align}
$\bb E \|M_1 \|^{\alpha + \eta} \iota(M_1)^{-\eta} < \infty$ and $\kappa_Q(\alpha + \eta) < \infty$. 
In addition, $\bb P (Q_1 = 0) < 1$. 
\end{conditionA}
Since the matrix norm is submultiplicative, $S_n = \log \| \Pi_n \|$ is a subadditive process
and it holds by the Furstenberg-Kesten theorem \cite{FK60} that $\lim_{n \to \infty} \frac{S_n}{n} = \Lambda'(0)$, $\bb P$-a.s. 
The identification of this limit is given by \cite[Theorem 6.1]{BDGM14} under the additional condition \ref{Condi_AP}, see below. 
Condition \ref{Condi_ExpMom}
allows in particular to introduce a change of probability measure $\bb P_{\alpha}$ for $S_n$ such that, 
$\lim_{n \to \infty} \frac{S_n}{n} = \Lambda'(\alpha)$, $\bb P_{\alpha}$-a.s.
Note that, due to the convexity of $\Lambda$, condition \ref{Condi_ExpMom} implies $\Lambda'(0) < 0 < \Lambda'(\alpha)$. 
This corresponds to the Cram\'er transform of a random walk with negative drift. 
In our result, we need the inverse drift of the transformed walk: 
\begin{align}\label{Def_rho}
\rho = \frac{1}{\Lambda'(\alpha)}. 
\end{align}

Denote by $\supp \mu$ the support of the measure $\mu$ and by $\Gamma_{\mu}$ 
the smallest closed semigroup of $\mathcal{M}$ generated by $\supp \mu$. 
A matrix $M \in \mathcal{M}$ is called allowable 
if every row and every column of $M$ has a strictly positive entry. 
We use the following allowability and positivity condition. 

\begin{conditionA}\label{Condi_AP}
Every $M \in\Gamma_{\mu}$ is allowable, 
and $\Gamma_{\mu}$ contains at least one strictly positive matrix. 
\end{conditionA}

According to the Perron-Frobenius theorem, a strictly positive matrix $M$ always has a unique dominant eigenvalue $\lambda_{M} >0$. 
We need the following non-arithmeticity condition on the measure $\mu$. 
\begin{conditionA}\label{Condi_NonArith}
The additive subgroup of $\mathbb{R}$ generated by the set
$\{ \log \lambda_{M}: M \in \Gamma_{\mu} \  \mbox{is strictly positive} \}$
is dense in $\mathbb{R}$. 
\end{conditionA}

To present the conditioned central limit theorem and large deviation asymptotics for the directional first passage time $\tau_u^y$,   
we will employ the following boundedness condition which is stronger than \ref{Condi_AP}. 
\begin{conditionA}\label{Condi_Kes_Weak}
There exists a constant $c > 1$ such that  for any $g = (g^{i,j})_{1\leq i, j \leq d} \in \supp \mu$
and $1 \leq j \leq d$, 
\begin{align}\label{Kes_Weak}
 \frac{\max_{1\leq i \leq d} g^{i,j} }{ \min_{1\leq i \leq d} g^{i,j}} \leq c.
\end{align}
\end{conditionA}

Our reason for imposing this condition is that we use the large deviation result for the coefficients of products of random matrices 
(see \eqref{LDPet_ScalarProduct} of Theorem \ref{PetrovThm})
which is only available under \ref{Condi_Kes_Weak}. 

Condition \ref{Condi_Kes_Weak} 
is weaker than the following one introduced by Furstenberg and Kesten \cite{FK60}: 
there exists a constant $c > 1$ such that  for any $g = (g^{i,j})_{1\leq i, j \leq d} \in \supp \mu$,
\begin{align}\label{Condi_Kesten}
 \frac{\max_{1\leq i, j\leq d} g^{i,j} }{ \min_{1\leq i,j\leq d} g^{i,j}} \leq c.
\end{align}

To formulate the subsequent results for the first passage times $\tau_u$ and $\tau_u^y$, 
we need to introduce some additional notation and discuss some results from the literature.  
For allowable $M \in \mathcal{M}$ and $x \in \bb S_+^{d-1}$, 
we write $M \cdot x= \frac{Mx}{|Mx|}$ for the projective action of $M$ on $\bb S_+^{d-1}$.   
Denote $G_n = M_n \ldots M_1$, $n \geq 1$. 
Under condition \ref{Condi_AP}, for any starting point $x \in \bb S_+^{d-1}$, 
the Markov chain $(G_n \cdot x)$ has a unique invariant probability measure $\nu$ on $\mathbb S^{d-1}_+$ 
(cf.\  \cite{Hen97, BDGM14}).

Let $s\in I_{\mu}$. We define the transfer operator $P_s$ and the conjugate transfer operator $P_{s}^{*}$ 
as follows: for any $x\in \bb S_+^{d-1}$ and any bounded measurable function $\varphi$  on $\bb S_+^{d-1}$, 
\begin{align}\label{transfoper001}
\! P_{s}\varphi(x)  = \int_{\Gamma_{\mu}}   |g x |^{s} \varphi( g \cdot x ) \mu(dg),  \quad 
P_{s}^{*}\varphi(x) = \int_{\Gamma_{\mu}}   |g^{\mathrm{T}}x|^{s} \varphi(g^{\mathrm{T}} \cdot x) \mu(dg),
\end{align}
where $g^\mathrm{T}$ is the transpose of $g$. 
Under \ref{Condi_AP}, 
the operator $P_s$ (resp.\ $P_{s}^{*}$) has a unique probability eigenmeasure $\nu_s$ (resp.\ $\nu^*_s$) on $\bb S_+^{d-1}$
satisfying $\nu_s P_s = \kappa(s)\nu_s$ (resp.\ $\nu^*_s P_{s}^{*} = \kappa(s)\nu^*_s$). 
Set, for $x\in \bb S_+^{d-1}$, 
\begin{align}\label{Def_r_s}
r_{s}(x)= \int_{\bb S_+^{d-1}} \langle x, y\rangle^s  \nu^*_s(dy),  \quad
r_{s}^*(x)= \int_{\bb S_+^{d-1}} \langle x, y\rangle^s  \nu_s(dy). 
\end{align}
Then $r_s$ (resp.\ $r^*_s$) is the unique (up to a scaling constant)  
strictly positive eigenfunction of $P_s$ (resp.\ $P_{s}^{*}$): $P_s r_s = \kappa(s)r_s$ (resp.\ $P_{s}^{*} r^*_s = \kappa(s)r^*_s$). 
In particular, $\nu$ is the eigenmeasure of $P_0$. 
See \cite[Proposition 3.1]{BDGM14}.

With this notation, the famous result of Kesten \cite{Kes73} reads: 
under \ref{Condi_ExpMom}, \ref{Condi_AP} and \ref{Condi_NonArith}, 
there exists a constant $\mathscr C >0$ such that 
as $u \to \infty$, 
\begin{align}\label{Exit_tail_V_n}
\bb P ( \tau_u^y < \infty ) = \bb P ( \langle y, V \rangle  > u )  \sim \mathscr C r_{\alpha}^*(y)  u^{-\alpha}, 
\end{align}
for the constant $\alpha > 0$ given by \ref{Condi_ExpMom}. 
This was extended by \cite{BDM02a} and \cite{CM18} to full multivariate regular variation, 
see Lemma \ref{Lem_CM_Weak_Conv} below for details.

\subsection{Laws of large numbers}

We state the conditioned weak laws of large numbers for the first passage times $\tau_u$ and $\tau_u^y$.

\begin{theorem}\label{Thm_LLN_Perpe}
Let $\rho$ be defined by \eqref{Def_rho}. 
Assume \ref{Condi_ExpMom}, \ref{Condi_AP} and \ref{Condi_NonArith}. 
Then, for any $\ee >0$, 
\begin{align*}
\lim_{u \to \infty} \bb P 
\left(  \left. \left| \frac{\tau_u}{ \log u } - \rho \right| > \ee  \, \right|  \tau_u < \infty \right)
= 0, 
\end{align*}
and uniformly in $y \in \bb S_+^{d-1}$, 
\begin{align*}
\lim_{u \to \infty} \bb P 
\left( \left. \left| \frac{\tau_u^y}{ \log u } - \rho \right| > \ee   \,  \right|  \tau_u^y < \infty \right)
= 0. 
\end{align*}
\end{theorem}

Indeed, a stronger version of Theorem \ref{Thm_LLN_Perpe} will be established in Theorem \ref{Thm_LLN_Perpe-stronger},
where $\ee$ is replaced by $b  \sqrt{ \frac{ \log \log u }{ \log u} }$ for some constant $b>0$. 
This gives a strong localization of the first passage times. 
Recall that $\rho$ is the inverse drift of $S_n = \log \| \Pi_n \|$ under the change of probability measure $\bb P_{\alpha}$, 
under which $S_n$ goes to infinity. The proof of Theorem \ref{Thm_LLN_Perpe} starts from the observation 
that $V_n$ behaves approximately like $e^{S_n}$. 
Note that the expected time $S_n$ reaches $\log u$ is $\rho \log u$. 
Hence the theorem states that the asymptotic of the hitting time of the perpetuity sequence 
is equal to that of the purely multiplicative process $e^{S_n}$.

\subsection{Central limit theorems} 
Denote by $\Phi$ the standard normal distribution function on $\bb R$. 
By \cite[Lemma 7.2]{BM16},
condition \ref{Condi_NonArith} ensures that $\sigma_{\alpha}: = \sqrt{ \Lambda''(\alpha) } >0$.
Now we state the conditioned central limit theorems for the first passage times $\tau_u$ and $\tau_u^y$. 

\begin{theorem}\label{Thm_CLT_Perpe}
Let $\rho$ be defined by \eqref{Def_rho}. 
If \ref{Condi_ExpMom}, \ref{Condi_AP} and \ref{Condi_NonArith} hold, then for any $t \in \bb R$, 
\begin{align}\label{CLT-tau-u}
\lim_{u \to \infty} \bb P 
\left(  \left. \frac{ \tau_u - \rho \log u  }{ \sigma_{\alpha} \rho^{3/2} \sqrt{\log u} } \leq t 
  \, \right|   \tau_u < \infty  \right) = \Phi(t). 
\end{align}
If \ref{Condi_ExpMom}, \ref{Condi_NonArith} and \ref{Condi_Kes_Weak} hold, 
then for any $t \in \bb R$ and $y \in \bb S_+^{d-1}$, 
\begin{align}\label{CLT-tau-u-y}
\lim_{u \to \infty} \bb P 
\left( \left. \frac{ \tau_u^y - \rho \log u }{ \sigma_{\alpha} \rho^{3/2} \sqrt{\log u} } \leq t 
   \,  \right|  \tau_u^y  < \infty  
\right) = \Phi(t). 
\end{align}
\end{theorem}

Theorem \ref{Thm_CLT_Perpe} together with \eqref{Exit_tail_V_n-intro} and \eqref{Exit_tail_V_n} implies that 
for any $t \in \bb R$ and $y \in \bb S_+^{d-1}$, as $u \to \infty$, 
\begin{align*}
\bb P  \left( \frac{ \tau_u - \rho \log u }{ \sigma_{\alpha} \rho^{3/2} \sqrt{\log u} } \leq t  \right) 
\sim  \mathscr C  u^{-\alpha}  \Phi(t),  \quad 
\bb P  \left( \frac{ \tau_u^y - \rho \log u }{ \sigma_{\alpha} \rho^{3/2} \sqrt{\log u} } \leq t  \right)
\sim  \mathscr C r_{\alpha}^*(y)  u^{-\alpha}  \Phi(t), 
\end{align*}
where $\mathscr C  >0$ is the constant given by \eqref{Exit_tail_V_n-intro}. 

To prove \eqref{CLT-tau-u}, we make use of Bahadur-Rao-Petrov type large deviation asymptotics for
the vector norm $|\Pi_n x|$, which has been established in \cite{XGL19a}, 
see \eqref{LDPet_VectorNorm} of Theorem \ref{PetrovThm}. 
For the proof of \eqref{CLT-tau-u-y}, we apply the precise large deviation asymptotics
for the coefficients $\langle y, \Pi_n x \rangle$,  which has been recently obtained in \cite{XGL22}, 
see \eqref{LDPet_ScalarProduct} of Theorem \ref{PetrovThm}. 
Further, the proof of Theorem \ref{Thm_CLT_Perpe} relies on the following 
result which is a refinement 
of the asymptotic in \eqref{Exit_tail_V_n-intro} by considering 
the joint law with the direction $V/|V|$. 
For any measurable set $A \subseteq \bb S_+^{d-1}$, denote by $\partial A$ the boundary of $A$.

\begin{lemma}[\cite{CM18}]\label{Lem_CM_Weak_Conv}
Assume \ref{Condi_ExpMom}, \ref{Condi_AP} and \ref{Condi_NonArith}. 
Then, for any measurable set $A \subseteq \bb S_+^{d-1}$ with $\nu_{\alpha} (\partial A) = 0$, 
\begin{align}\label{Asymp-Men}
\lim_{u \to \infty}  u^{\alpha} \,  \bb P \left( \frac{V}{|V|} \in A, \,  |V| > u \right)
= \mathscr{C} \,  \nu_{\alpha} (A). 
\end{align}
Moreover, 
\begin{align}\label{Weak_Conver_V}
\lim_{t \to \infty}  e^{\alpha t}
\bb E  \Big[ r_{\alpha} \Big( \frac{V}{|V|} \Big)  \mathds 1_{ \left\{ |V| \geq  e^t  \right\} }  \Big]
= \mathscr{C} \nu_{\alpha}(r_{\alpha}). 
\end{align}
\end{lemma}

The asymptotic \eqref{Asymp-Men} is a direct consequence of \cite[Theorem 2.4]{CM18}, see (2.13) in \cite{CM18}. 
According to the Portmanteau theorem, \eqref{Weak_Conver_V} follows from \eqref{Asymp-Men}.

\subsection{Growth of moments}\label{subsect:momentgrowth}
If $s > \alpha$, then $\bb E (|V|^{s}) = \infty$. In this section we study the growth rate of $\bb E (|V_n|^s)$. Recalling that the perpetuity sequence $V_n$ has the same marginal distribution as the forward process $V_n^*$, the following result is valid as well when replacing $V_n$ by $V_n^*$. Indeed, the lemma will be proved by considering the forward process.

\begin{lemma}\label{Lem_Martin}
Let $s > \alpha$ with $\alpha$ defined by \eqref{Def_alpha}. Assume that $\bb E (|Q_1|^s) < \infty$. Then, 
the following limits exist, are strictly positive and finite:
\begin{align}\label{Limit-Vn-Vnstar}
\lim_{n \to \infty} \frac{1}{ \kappa(s)^n }  \bb E \left[ |V_n|^s r_s \left( \frac{V_n}{|V_n|} \right) \right]
    \in (0, \infty). 
\end{align} 
In particular, the sequence $\bb E (|V_n^s|) / \kappa(s)^n$ is bounded from above and below. 
\end{lemma}

The proof of Lemma \ref{Lem_Martin} is given in Section \ref{Sec-Pf-Thms-5-8}. 

\subsection{Precise large deviations}

For any $s > \alpha$ with $\alpha$ given by \eqref{Def_alpha}, let 
\begin{align}\label{Def_C_beta}
\varkappa_s = \frac{ \sqrt{\Lambda'(s)} }{s \sigma_s \nu_s(r_s) \sqrt{2 \pi}} 
 \lim_{n \to \infty} \frac{1}{\kappa(s)^n} 
   \bb E \left[ |V_n|^s r_s \left( \frac{V_n}{|V_n|} \right)  \right]. 
\end{align}
As mentioned before, by \cite[Lemma 7.2]{BM16}, condition \ref{Condi_NonArith} ensures that $\sigma_s >0$ for any $s\in I_{\mu}$.
By  Lemma \ref{Lem_Martin}, $\varkappa_s$ is strictly positive and finite.  

Now we introduce the rate function which is used to describe the polynomial rates of decay  
in the large deviation asymptotics for $\tau_u$ and $\tau_u^y$. 
The Fenchel-Legendre transform of $\Lambda = \log \kappa$ is defined by
\begin{align*}
\Lambda^{\ast}(q)=\sup_{s\in \bb R}\{sq-\Lambda(s)\},   \quad  q\in\Lambda'(I_{\mu}). 
\end{align*}
It follows that $\Lambda^*(q)=s q - \Lambda(s)$ if $q=\Lambda'(s)$ for some $s\in I_{\mu}$.  
For any $\beta = \frac{1}{\Lambda'(s)}$ with $s \in I_{\mu}^{\circ}$, we define $I(\beta) = \beta \Lambda^* ( \frac{1}{\beta} )$
and it holds that
\begin{align}\label{Def_I_beta_s}
I(\beta) = \frac{1}{\Lambda'(s)} \big( s \Lambda'(s) - \Lambda(s)  \big)
        = s - \frac{ \Lambda(s) }{ \Lambda'(s) }. 
\end{align}
Denote $\frac{1}{2} I_{\mu}^{\circ} = \{ s \geq 0: 2s \in I_{\mu}^{\circ} \}$. 
For $l, u \geq 0$ and $\beta = \frac{1}{\Lambda'(s)}$ with $s \in I_{\mu}^{\circ}$,
we denote 
\begin{align}\label{def-chi-C-beta-l-u}
\chi_{\beta, l} (u) = (\beta - l) \log u - \floor{(\beta - l) \log u} \quad  \mbox{and}  \quad 
\mathscr C_{\beta, l} (u) = \varkappa_s \kappa(s)^{- \chi_{\beta, l} (u)}. 
\end{align}
Since $\chi_{\beta, l} (u) \in [0, 1)$, we have that $\mathscr C_{\beta, l} (u)$ is strictly positive and finite.  

Now we state the precise large deviation results for the first passage times $\tau_u$ and $\tau_u^y$, 
where a vanishing perturbation $l$ on $\beta$ is allowed.

\begin{theorem}\label{Thm_LD_Perpe_Petrov}
Let $\beta \in (0, \rho)$ with $\rho$ defined by \eqref{Def_rho}. 
Assume that there exists $s \in \frac{1}{2} I_{\mu}^{\circ}$ such that $\beta = \frac{1}{\Lambda'(s)}$. 
If \ref{Condi_ExpMom}, \ref{Condi_AP} and \ref{Condi_NonArith} hold, 
then for any $\epsilon >0$, as $u \to \infty$, uniformly in $l \in [0, (\log u)^{-\epsilon}]$, 
\begin{align}\label{Petrov_LD_tau}
\bb P (\tau_u \leq (\beta - l) \log u) 
\sim \frac{ \mathscr C_{\beta, l} (u)  }{\sqrt{\log u}} 
  u^{-I(\beta - l)}. 
\end{align}
If \ref{Condi_ExpMom}, \ref{Condi_NonArith} and \ref{Condi_Kes_Weak} hold, 
then for any $\epsilon >0$, as $u \to \infty$, uniformly in $y \in \bb S_+^{d-1}$ and $l \in [0, (\log u)^{-\epsilon}]$, 
\begin{align}\label{Petrov_LD_tau_y}
\bb P \left( \tau_u^y \leq (\beta - l) \log u \right) 
\sim  r_s^*(y) \frac{ \mathscr C_{\beta, l} (u) }{\sqrt{\log u}} 
  u^{-I(\beta - l)}. 
\end{align}
\end{theorem}

For $d=1$, 
our results \eqref{Petrov_LD_tau} and \eqref{Petrov_LD_tau_y}
partially extend that of 
Buraczewski, Collamore, Damek and Zienkiewicz \cite{BCDZ16} by considering a vanishing perturbation $l$ on $\beta$
(however, we do not include the case of negative $Q$). 
This extension turns out to be very useful to deduce new local limit theorems of $\tau_u$ and $\tau_u^y$,
see Section \ref{sec-LLT} below.

In particular, the asymptotic \eqref{Petrov_LD_tau} clearly implies the following large deviation principle which is also new:
uniformly in $l \in [0, (\log u)^{-\epsilon}]$, 
\begin{align*}
\lim_{u \to \infty} \frac{1}{\log u}  \log  \bb P \Big( \tau_u \leq (\beta - l) \log u \Big) 
= -I(\beta). 
\end{align*}
A similar consequence of \eqref{Petrov_LD_tau_y} can also be obtained.

\begin{remark}
Taking $l = 0$ in Theorem \ref{Thm_LD_Perpe_Petrov}, we get that, as $u \to \infty$, 
\begin{align}
\bb P (\tau_u \leq \beta \log u) 
& \sim  \frac{ \mathscr C_{\beta, 0} (u) }{\sqrt{\log u}} u^{-I(\beta)},    \label{Bahadur_Rao_LD_tau} \\
\bb P \left( \tau_u^y \leq \beta \log u \right) 
& \sim  r_s^*(y) \frac{ \mathscr C_{\beta, 0} (u) }{\sqrt{\log u}} u^{-I(\beta)}.   \label{Bahadur_Rao_LD_tau_y}
\end{align}
\end{remark}

\begin{remark}
We will see in Lemma \ref{Lem_Expan_Rate} that the rate function $I(\beta - l)$ has the asymptotic expansion 
with respect to the perturbation $l$ in a small neighborhood of $0$:
\begin{align}\label{Expan_I_beta_l}
I(\beta - l)  = I(\beta) + \frac{l}{\beta - l} I(\beta) + \frac{sl}{\beta (\beta - l)^2}  
   + \frac{l^2}{ 2 \sigma_s^2 \beta^2 (\beta - l)^3 } 
   - \frac{l^3}{\sigma_s^3 \beta^3 (\beta - l)^4 } \xi_s \left( \frac{l}{ \sigma_s \beta (\beta - l) } \right),
\end{align}
where $\xi_s$ is the Cram\'{e}r series with  
coefficients given in terms of the cumulant generating function $\Lambda = \log \kappa$ at the point $s$, cf.\ \eqref{def-xi-s-t}. 
\end{remark}

\begin{remark}\label{Rem-Upper-LD}
In \cite{BCDZ16, BDZ18}, interesting results on upper large deviations are also obtained and it turns out that the story is completely different.
There are examples given in \cite[Theorem 2.3]{BCDZ16}  and \cite[Theorem 1.14]{BDZ18} 
showing that the decay 
may occur at a slower rate, such as
$u^{- I(\beta) + \delta}$ for some $\delta > 0$. 
Extending these results to the multivariate setting remains beyond the current scope of research. 
\end{remark}

\subsection{Local limit theorems}\label{sec-LLT}

The vanishing perturbation $l$ in \eqref{Petrov_LD_tau} and \eqref{Petrov_LD_tau_y} turns out to be very useful
for deducing precise local limit theorems for the first passage times $\tau_u$ and $\tau_u^y$. 
Note that from \eqref{Expan_I_beta_l} we have 
\begin{align*}
I'(\beta) = - \frac{1}{\beta} I(\beta) - \frac{s}{\beta^3}. 
\end{align*} 
Denote $\bb R_- = (-\infty, 0]$ and $\bb Z_+ = \{1, 2, 3, \ldots \}$.

\begin{theorem}\label{Cor_LLT_LD_tau}
Let $\beta \in (0, \rho)$ with $\rho$ defined by \eqref{Def_rho}. 
Assume that there exists $s \in \frac{1}{2} I_{\mu}^{\circ}$ such that $\beta = \frac{1}{\Lambda'(s)}$. 
Let $a \in \bb R_-$ and $m \in \bb Z_+$ be such that $a + m \leq 0$.  
If \ref{Condi_ExpMom}, \ref{Condi_AP} and \ref{Condi_NonArith} hold, 
then for any $\epsilon >0$, as $u \to \infty$, uniformly in $l \in [0, (\log u)^{-\epsilon}]$, 
\begin{align*}
 \bb P \Big(\tau_u - (\beta - l) \log u  \in (a, a + m ] \Big)  
 \sim e^{a I'(\beta)} \left( e^{m I'(\beta)} - 1 \right)
    \frac{ \mathscr C_{\beta, l} (u) }{\sqrt{\log u}} 
    u^{-I(\beta - l)}. 
\end{align*}
If \ref{Condi_ExpMom}, \ref{Condi_NonArith} and \ref{Condi_Kes_Weak} hold, 
then for any $\epsilon >0$, as $u \to \infty$, uniformly in $y \in \bb S_+^{d-1}$ and $l \in [0, (\log u)^{-\epsilon}]$, 
\begin{align*}
 \bb P \Big(\tau_u^y - (\beta - l) \log u  \in (a, a + m ] \Big)  
 \sim  r_s^*(y) e^{a I'(\beta)} \left( e^{m I'(\beta)} - 1 \right)
    \frac{ \mathscr C_{\beta, l} (u) }{\sqrt{\log u}} 
    u^{-I(\beta - l)}. 
\end{align*}
\end{theorem}

Up to our knowledge, Theorem \ref{Cor_LLT_LD_tau} is new even in the one-dimensional case $d=1$. 
In particular, taking  $a = -1$, $m = 1$ and $l = 0$ in Theorem \ref{Cor_LLT_LD_tau}, we get
\begin{align}
& \bb P \Big( \tau_u = \floor{\beta \log u} \Big)  
 \sim  \left( 1 - e^{- I'(\beta)}  \right)
    \frac{ \mathscr C_{\beta, 0} (u) }{\sqrt{\log u}} 
    u^{-I(\beta)},   \label{Asym-pointwise-tau-u} \\
& \bb P \Big(\tau_u^y = \floor{\beta \log u}  \Big)  
 \sim  r_s^*(y) \left( 1 - e^{- I'(\beta)}  \right)
    \frac{ \mathscr C_{\beta, l} (u) }{\sqrt{\log u}} 
    u^{-I(\beta)}.  \label{Asym-pointwise-tau-u-y}
\end{align}
In the one-dimensional case $d=1$, by applying a different method, 
the asymptotic \eqref{Asym-pointwise-tau-u} has been recently established by Buraczewski, Damek and Zienkiewicz \cite{BDZ18}. 
This result is even valid for the general case $Q \in \bb R$, 
but under the additional assumptions that $M_1$ has a density $f(t)$ satisfying $f(t) \leq c (1 + t)^{-D}$ for some $D> 1 + \alpha$
and that $\bb P (M_1 \in dt, Q_1 \in ds) \leq f(t) dt \, \eta(ds)$ for some probability measure $\eta$ on $\bb R$. 
Note moreover that Theorem 1.6 in \cite{BDZ18} covers the one-dimensional case also for 
$\beta> \rho$, see also Remark \ref{Rem-Upper-LD}. 

\subsection{Examples}\label{sect:examples}

\begin{example}[Asymptotic independence of financial time series]\label{exa:garch1}
	Consider GARCH(1,2) processes $X_t$ and $Y_t$ driven by {\it the same} i.i.d.\ process $Z_t$, i.e.,
	$$ X_t = \sigma_t Z_t, \qquad Y_t = \eta_t Z_t, \qquad t \ge 0$$
	for an i.i.d.\ sequence $Z_t$ with mean zero and unit variance, where the conditional variance processes $\sigma_t$ and $\eta_t$ satisfy the following recursions: $\sigma_{-1}=\sigma_{-2}=\eta_{-1}=\eta_{-2}=0$ and
	$$ \sigma_t^2 ~=~ a_0+ a_1 X_{t-1}^2 + b_1 \sigma_{t-1}^2 + b_2 \sigma_{t-2}^2, \qquad \eta_t^2 ~=~ c_0 + c_1 Y_{t-1}^2 + d_1 \eta_{t-1}^2 + d_2 \eta_{t-2}^2,$$
	for positive coefficients $a_i, b_i, c_i, d_i$, $i \in \{1,2\}$ satisfying $a_1+b_1+b_2<1$ as well as $c_1+d_1+d_2<1$.
	
	Upon introducing the nonnegative matrices and vectors
	$$ M_{t} ~:=~ \begin{pmatrix}
		b_1 + a_1 Z_{t-1}^2 & b_2 \\
		1 & 0 
	\end{pmatrix}, \quad Q_t ~:=~ \begin{pmatrix}
	a_0 \\ 0
\end{pmatrix}, \qquad \widetilde{M}_{t} ~:=~ \begin{pmatrix}
d_1 + c_1 Z_{t-1}^2 & d_2 \\
1 & 0 
\end{pmatrix}, \quad \widetilde{Q}_t ~:=~ \begin{pmatrix}
c_0 \\ 0
\end{pmatrix}$$
we have the multivariate recursions
	$$ \begin{pmatrix}
		\sigma_t^2 \\ \sigma_{t-1}^2 
	\end{pmatrix} ~=~ M_t \begin{pmatrix}
\sigma_{t-1}^2 \\ \sigma_{t-2}^2 
\end{pmatrix} + Q_t, \qquad \begin{pmatrix}
\eta_t^2 \\ \eta_{t-1}^2 
\end{pmatrix} ~=~ \widetilde{M}_t \begin{pmatrix}
\eta_{t-1}^2 \\ \eta_{t-2}^2 
\end{pmatrix} + \widetilde{Q}_t$$

The (law of the) random variable $\sigma^*$ and $\eta^*$, given by the corresponding perpetuity series
$$ \sigma^* := \langle e_1, Q_1 +\sum_{n=2}^\infty M_1 \cdots M_{n-1} Q_n \rangle, \qquad \eta^* := \langle e_1, \widetilde{Q}_1 +\sum_{n=2}^\infty \widetilde{M}_1 \cdots \widetilde{M}_{n-1} \widetilde{Q}_n \rangle, $$
is a stationary distribution for $\sigma_t$ and $\eta_t$, respectively. By an application of Kesten's result (see \cite{BDM02a}, \cite[Example 4.4.13]{BDM16} for details), there are constants $\alpha, \beta >0$ such that $\sigma^*$ and $\eta^*$ are regularly varying with indices $\alpha$ and $\beta$, respectively: It holds $\bb P(\sigma^*>u) \sim C_1 u^{-\alpha}$, $\bb P(\eta^*>u) \sim C_2 u^{-\beta}$, or, equivalently
$$ \lim_{u \to \infty} u \cdot\bb P (\sigma^* > u^{1/\alpha}) = C_1, \quad \lim_{u \to \infty} u \cdot \bb P(\eta^* > u^{1/\beta})  =  C_2. $$
Does this -- due to the dependence on {\it the same} random input $Z_t$ -- imply asymptotic dependence, namely, does it hold
$$ \lim_{u \to \infty} u \cdot \bb P\big( \sigma^*>u^{1/\alpha}, \ \eta^* > u^{1/\beta} \big) = C_3 $$
for a positive constant $C_3$?	
For the one-dimensional case (meaning a univariate recursion, which includes GARCH(1,1)), it was proved in \cite{MW2022} that in general, $C_3=0$, meaning that the processes are {\it asymptotically independent}. The proof in \cite{MW2022} is based on the one-dimensional version of Theorem \ref{Thm_LLN_Perpe}: one makes use of the fact that the first passage time $\tau_u$ pertaining to the sequence converging to $\sigma^*$  is well localized around $\rho_\sigma \log u$, and that one can decompose $\eta^*$ into the contributions up to time $\rho_\sigma \log u$, which are small due to the fact that - under general assumptions - $\rho_\sigma$ and $\rho_\eta$ are different; and the contributions after time $\rho_\sigma \log u$, which are basically independent. See the proof of Theorem 5.1 in \cite{MW2022} for details. In the higher-order setting described above, this leads to the comparison of first passage times $\tau_{u}^{e_1}$ and $\widetilde{\tau}_{u}^{e_1}$ for the perpetuity sequences converging to $\eta^*$ and $\sigma^*$, respectively.
Therefore, our results can be used to derive an analogous result about asymptotic independence for higher order GARCH models.

Note that, by an application of Breiman's lemma \cite{Bre65}, the observed processes $X^2$ and $Y^2$ are regulary varying as well with indices $\alpha$ and $\beta$, respectively. Further, by Lemmata 2.6 and 2.10 in \cite{BDM12} (applied to $X^{2\alpha}$, $Y^{2\beta}$), asymptotic independence of the  stationary distributions of the volatility processes implies asymptotic independence for the stationary distributions of the observed processes. 
We have considered here for notational simplicity the extremal setting where both processes rely on the same noise, instead of the more realistic setup of correlated noise sequences $Z_t^X$ and $Z_t^Y$. However, asymptotic independence in our setup implies asymptotic independence in the less dependent setting. To wit, if $Z_t^X$ and $Z_t^Y$ are Gaussian, we can decompose $Z_t^Y=cZ_t^X + Z'_t$, where $Z'_t$ is independent of $Z_t$ and $X$ and recall that independent processes are in particular asymptotically independent. 
\end{example}

\begin{example}[{Perpetuities, cf.\  \cite[Example 3.4]{CV13}}] The notion {\it perpetuity} comes from the life insurance context: To calculate the amount of money necessary to pay a {\it perpetual} pension, one has to sum the discounted obligations. Let $R_k$ denote the interest rate in year $k$, then $A_k=1/(1+R_k)$ gives the discount factor. Note that we have encountered negative interest rates, so $A_k>1$ is possible. If $B_n$ denotes the obligation in year $n$, then
	$$ P_n = \sum_{k=1}^n A_1 \cdots A_{k-1}B_k$$
	is the capital needed to pay the pension for $n$ years. Then $\tau_u$ is the number of years the pension can be paid when the initial capital is $u$ (time to ruin). So our results provide precise estimates for the possible duration of the payment.

The multivariate setting may become relevant if different components of the vectors $P_n$ and $B_n$ correspond to savings and obligations, respectively, in different currencies. Then the diagonal entries of $A$ would correspond to discount rates for different currencies (countries), while off-diagonal entries correspond to a shifting of capital, incorporating exchange rates.
\end{example}

\begin{example}[Multitype branching processes with immigration in random environment] Consider a $d$-type branching process with immigration in an i.i.d.\  random environment $\mathcal{E}$ as in \cite{Key87, Roi07}. Let  $Z(t)=(Z^1(t), \dots, Z^d(t)) \in \bb{N}^d$ be the offspring born to immigrants who arrived between time $0$ and $t-1$ inclusive, here $Z^i(t)$ is the offspring of type $i$, $1 \leq i \leq d$. Then the quenched expectation $Y(t)$ of $Z(t)$  satisfies
$$ Y(t)~=~\bb{E} \big[ Z(t) \, \big| \, \mathcal{E} \big] ~=~ \sum_{k=0}^{t-1} I(k)M(t-k+1) \cdots M(t)$$
where $I(k)$ is the vector of the expected number of immigrants at time $k$, given the environment, and $M(\ell)=M(\ell)^{i,j}$ with $M(\ell)^{i,j}$ being the quenched expected number of type $j$-offspring born to type $j$-parents at time $\ell$,	see \cite[Eq.s (5), (15)]{Key87}.
For each $t$, the random variable $Y(t)$ or rather $Y(t) - \bb E (Z(t))$ describes how the expected population size fluctuates due to the environment. 
Note that $Y(t)$ is considered as a row vector, $Y(t)^\top$ takes the form of a forward process $V_n^*$ as defined in \eqref{Def_Yn02}. Using  that $V_n$ and $V_n^*$ have the same marginal distribution, we see that
$$ \bb P(\tau_u^{e_i} \leq t) ~=~ \bb P(\langle e_i,V_t\rangle >u)~=~ \bb P(\langle e_i,V_t^*\rangle >u) ~=~ \bb P (Y^i(t) > u).$$
Hence, Theorem \ref{Thm_LD_Perpe_Petrov} provides precise large deviation estimates for the quenched expectation of $Z^i(t)$ for each type $1 \leq i \leq d$. 
\end{example}

\begin{example}[Growth models of Sigma-Pi-type]
	In a similar way as the previous example, the forward process (cf. \eqref{Def_Yn02})
	\begin{align*}
		V_n^* = Q_n + M_n Q_{n-1} + \cdots + (M_n \ldots M_2) Q_1,  \quad  n \geq 1. 
	\end{align*}
	can be used to describe e.g. the growth of the total biomass of a system; where the components of $Q_n$ describe the mass of different species of bacteria entering the system at time $n$, while the ($\bb R_\ge$-valued) components of $M_k$ describe growth factors from time $k \to k+1$; as well as (positive) interaction effects, see \cite{STST17} for a discussion, as well as references therein. Hence, our results give precise large deviation estimates for $\bb P(\scal{V_n^*}{e_j}>u)$, for any $1\le j  \le d$. 
\end{example}

\section{Proof of Theorem \ref{Thm_LLN_Perpe}}
 
The main goal of this section is to establish Theorem \ref{Thm_LLN_Perpe} 
on the conditioned weak laws of large numbers for $\tau_u$ and $\tau_u^y$. 
We start with the following inequality, which holds uniformly in $s$ and improves \cite[Corollary 4.6]{BDGM14} 
established for fixed $s$. 
 Such an improvement plays an important role in the sequel. 

\begin{lemma}\label{Lem_kappa}
Assume \ref{Condi_ExpMom} and \ref{Condi_AP}. 
Then, for any compact set $K \subset I_{\mu}^{\circ}$, 
there exists a constant $c = c_K < \infty$ such that for any $s \in K$ and $n \geq 1$, 
\begin{align}\label{Ine_Kappa}
\bb E (\| \Pi_n \|^s) \leq c \, \kappa(s)^n. 
\end{align}
\end{lemma}

\begin{proof}
Since $P_s r_s = \kappa(s) r_s$, we get $P_s^n r_s = \kappa(s)^n r_s$ and hence
\begin{align}\label{Equation-Ps-rs}
\bb E \Big[ | G_n  x|^s r_s(G_n \!\cdot\! x) \Big] = \kappa(s)^n r_s(x), 
\end{align}
where $G_n = M_n \ldots M_1$. 
By \cite[Proposition 4.13]{BDGM14}, the map $s \mapsto r_s$ is continuous 
on $I_{\mu}^{\circ}$ with respect to the supremum norm on $\bb S_+^{d-1}$,  
hence it is uniformly continuous on the compact set $K \subset  I_{\mu}^{\circ}$.
This, together with the strict positivity of $r_s$, implies that 
there exist constants $c, C >0$ such that $c < r_s(x) < C$ for any $x \in \bb S_+^{d-1}$ and $s \in K$. 
Therefore, by \eqref{Equation-Ps-rs}, there exists a constant $c>0$ such that for any $x \in \bb S_+^{d-1}$ and $s \in K$, 
\begin{align*}
\kappa(s)^n \geq c \bb E \big( | G_n x|^s \big). 
\end{align*}
By \cite[Lemma 4.5]{BDGM14}, for any $x$ belonging to the interior of  $\bb S_+^{d-1}$, there exists a constant $c>0$ such that 
$\inf_{M \in \Gamma_{\mu}} \frac{|M x|}{\|M \|} >c$, so that $\bb E ( | G_n x|^s ) \geq c \bb E ( \| G_n \|^s ) = c \bb E ( \| \Pi_n \|^s )$
and the desired result follows. 
\end{proof}

The following result gives useful estimates for the tail distribution of the perpetuity sequence $V_n$,
which will be applied frequently in the sequel.  
Recall that $\kappa$ is defined by \eqref{def-kappa}, $\Lambda(s) = \log \kappa(s)$ and $\kappa_Q(s) = \bb E (|Q_1|^s)$. 

\begin{lemma}\label{Lem_Yn_u}
Let $s_0 \geq \alpha$ with $\alpha$ defined by \eqref{Def_alpha}. 
Assume \ref{Condi_ExpMom}, \ref{Condi_AP} 
and that there exists $\ee_0 \in (0,1)$ such that $\kappa(s_0 + t) < \infty$ and $\kappa_Q(s_0 + t) < \infty$ for all $|t| \leq \ee_0$. 
Then there exist constants $\ee > 0$ and $c >0$ such that 
for any $s \in [s_0, s_0 + \ee)$, $u >0$, $n \geq 1$ and $y \in \bb S_+^{d-1}$, 
\begin{align}\label{Bound_y_Vn_ttt02}
\bb P ( \langle y, V_n \rangle  > u) \leq 
\bb P (|V_n| > u) \leq c \,  n^{1+ s + \ee} 
\exp \left\{ n \left[ \ee \Lambda'(s) + \ee^2 \Lambda''(s)  \right] \right\}  \kappa(s)^n u^{-(s + \ee)}. 
\end{align}
Here, the constant $c$ does not depend on $\ee$. 
\end{lemma}

\begin{proof}
The proof follows from that of \cite[Lemma 3.1]{BCDZ16}. 
Since $\langle y, V_n \rangle \leq |V_n|$, we only need to prove the second inequality in \eqref{Bound_y_Vn_ttt02}. 
As in (3.4) of \cite{BCDZ16},
by the independence of $\Pi_{j-1}$ and $Q_j$, and Lemma \ref{Lem_kappa}, 
there exists a constants $c > 0$ such that for any $s \in [s_0, s_0 + \ee)$, 
\begin{align}\label{Pf-Vn-aa}
\bb P (|V_n| > u) 
\leq \sum_{j=1}^n \bb P \left( |\Pi_{j-1} Q_j| > \frac{u}{n} \right)   
  \leq  \left( \frac{u}{n} \right)^{ - (s+\ee) }  \sum_{j=1}^n \bb E  \left( |\Pi_{j-1} Q_j|^{s + \ee} \right)
  \leq  c n^{s + \ee} u^{- (s+ \ee)}  \sum_{j=1}^n  e^{ j \Lambda(s + \ee) }. 
\end{align}
Following again the proof of \cite[Lemma 3.1]{BCDZ16}, 
one has $\Lambda(s + \ee) \leq \Lambda(s) + \Lambda'(s) \ee +  \Lambda''(s) \ee^2$
when $\ee$ is sufficiently small. 
Since $s \geq \alpha$, we have $\Lambda(s) \geq 0$, $\Lambda'(s) >0$ and $\Lambda''(s) >0$. Hence, uniformly in $s \in [s_0, s_0 + \ee)$, 
\begin{align*}
\sum_{j=1}^n  e^{ j \Lambda(s + \ee) }
\leq c \,  n  \exp \left\{ n \big[ \Lambda(s) + \Lambda'(s) \ee +  \Lambda''(s) \ee^2 \big]  \right\}
= c \, n  \,  \kappa(s)^n  \exp \left\{ n \big[ \Lambda'(s) \ee +  \Lambda''(s) \ee^2 \big]  \right\}. 
\end{align*}
Substituting this into \eqref{Pf-Vn-aa} concludes the proof of \eqref{Bound_y_Vn_ttt02}. 
\end{proof}

Instead of the inequality, if we use the equality
 $\Lambda(s + \ee) = \Lambda(s) + \Lambda'(s) \ee + \frac{1}{2} \Lambda''(s_1) \ee^2$
for some $s_1 \in (s, s + \ee)$, 
then we get the following result which will be useful to establish Theorem \ref{Thm_LD_Perpe_Petrov}. 

\begin{remark}\label{Rem-Inequality-Vn}
Let $s_0 \geq \alpha$ with $\alpha$ defined by \eqref{Def_alpha}. 
Assume \ref{Condi_ExpMom}, \ref{Condi_AP} 
and that there exists $t > 0$ such that $\kappa(s_0 + t') < \infty$ and $\kappa_Q(s_0 + t') < \infty$ for all $|t'| \leq 2t$. 
Then there exists a constant $c >0$ such that 
for any $s \in [s_0, s_0 + t)$, $u >0$, $n \geq 1$ and $y \in \bb S_+^{d-1}$, the following holds. 
There exists $s_1 \in (s, s + t)$ such that 
\begin{align}\label{reamrk-Vn-u01}
\bb P ( \langle y, V_n \rangle  > u) \leq 
\bb P (|V_n| > u) \leq c \,  n^{1+ s + t} 
\exp \left\{ n \left[ t \Lambda'(s) + \frac{1}{2} t^2 \Lambda''(s_1)  \right] \right\}  \kappa(s)^n u^{-(s + t)}.  
\end{align}
Here the constant $c$ does not depend on $t$. 
\end{remark}

The following result gives an asymptotic expansion of the rate function $I(\beta - l)$ 
with respect to $l$ in a small neighborhood of $0$. 
Denote $\gamma_{s,k} = \Lambda^{(k)} (s)$ and let $\xi_s$ be the Cram\'er series (cf.\ \cite{Pet75, XGL19b}) of $\Lambda(s)$ given by
\begin{align}\label{def-xi-s-t}
\xi_s (t) = \frac{\gamma_{s,3} }{ 6 \gamma_{s,2}^{3/2} }  
  +  \frac{ \gamma_{s,4} \gamma_{s,2} - 3 \gamma_{s,3}^2 }{ 24 \gamma_{s,2}^3 } t
  +  \frac{\gamma_{s,5} \gamma_{s,2}^2 - 10 \gamma_{s,4} \gamma_{s,3} \gamma_{s,2} + 15 \gamma_{s,3}^3 }{ 120 \gamma_{s,2}^{9/2} } t^2
  +  \ldots.
\end{align}

\begin{lemma}\label{Lem_Expan_Rate}
Assume \ref{Condi_ExpMom}, \ref{Condi_AP} and that there exists $s \in I_{\mu}^{\circ}$ 
such that $\beta = \frac{1}{\Lambda'(s)} \in (0, \rho)$ with $\rho$ defined by \eqref{Def_rho}. 
Then there exists a constant $\delta>0$ such that for any $|l| \leq \delta$,
\begin{align*}
I(\beta - l) = I(\beta) + \frac{l}{\beta - l} I(\beta) + \frac{sl}{\beta (\beta - l)^2} + h_s(l), 
\end{align*}
where the function $h_s$ is linked to the Cram\'{e}r series $\xi_s$ by
\begin{align*}
h_s(l) = \frac{l^2}{ 2 \sigma_s^2 \beta^2 (\beta-l)^3 } 
  - \frac{l^3}{\sigma_s^3 \beta^3 (\beta - l)^4 } \xi_s \left( \frac{l}{ \sigma_s \beta (\beta-l) } \right).
\end{align*}
\end{lemma}

\begin{proof}
By the definition of $I$, we write
\begin{align*}
I(\beta - l) = \frac{1}{\beta - l} \Lambda^* \left( \frac{1}{\beta - l} \right) 
 = \frac{1}{\beta - l} \Lambda^* \left( \frac{1}{\beta} + \frac{l}{ \beta(\beta - l) }  \right).
\end{align*}
It was shown in \cite{XGL19a} that for any $|l_1| \leq \delta$, 
\begin{align*} 
\Lambda^* \left( \frac{1}{\beta} + l_1  \right)
   = \Lambda^{*} \left( \frac{1}{\beta} \right) + sl_1 
      + \frac{l_1^2}{2 \sigma_s^2} - \frac{l_1^3}{\sigma_s^3} \xi_s \left( \frac{l_1}{\sigma_s} \right).  
\end{align*}
Taking $l_1 = \frac{l}{ \beta(\beta - l) }$, we get
\begin{align*}
 \frac{1}{\beta - l} \Lambda^* \left( \frac{1}{\beta} + \frac{l}{ \beta(\beta - l) }  \right)  
& =  \frac{1}{\beta - l} \left[ \Lambda^{*} \left( \frac{1}{\beta} \right) + \frac{sl}{ \beta(\beta - l) } 
    +  \frac{l^2}{ 2 \sigma_s^2 \beta^2 (\beta - l)^2 } 
    - \frac{l^3}{\sigma_s^3 \beta^3 (\beta - l)^3 } \xi_s \left( \frac{l}{ \sigma_s \beta (\beta - l) } \right)   \right]   \notag\\
& = \frac{\beta}{\beta - l} I(\beta) + \frac{sl}{\beta (\beta - l)^2} + h_s(l), 
\end{align*}
as desired. 
\end{proof}

From Lemma \ref{Lem_Yn_u} we deduce the following asymptotics for the first passage times $\tau_u$ and $\tau_u^y$,
which will also be applied to establish Theorem \ref{Thm_LD_Perpe_Petrov} on the large deviation results for $\tau_u$ and $\tau_u^y$. 

\begin{lemma}\label{Lem_tau_Low}
Let $\beta \in (0, \rho)$ with $\rho$ defined by \eqref{Def_rho}. 
Assume \ref{Condi_ExpMom}, \ref{Condi_AP} and 
that there exists $\ee > 0$ such that $\kappa(s_0 + t) < \infty$ and $\kappa_Q(s_0 + t) < \infty$ for all $|t| \leq \ee$. 
Then, for any $a > \frac{3 + 2s}{2\Lambda(s)}$ and any sequence $(l_u)$ of positive numbers satisfying $\lim_{u \to \infty} l_u = 0$,  
we have, as $u \to \infty$,
uniformly in $l \in [0, l_u]$ and $y \in \bb S_+^{d-1}$, 
\begin{align}\label{Bound_tau_y_jj02}
\bb P \Big( \tau_u \leq  \upzeta_u  \Big) 
= \bb P ( |V_{\upzeta_u}| > u ) 
= o \left( \frac{u^{-I(\beta - l)}}{ \sqrt{ \log u } } \right), 
\qquad 
\bb P \Big( \tau_u^y \leq  \upzeta_u  \Big) 
= \bb P ( \langle y, V_{\upzeta_u} \rangle > u ) 
= o \left( \frac{u^{-I(\beta - l)}}{ \sqrt{ \log u } } \right), 
\end{align}
where $\upzeta_u = \floor{ (\beta - l) \log u - a \log \log u }$.
\end{lemma}

\begin{proof}
It suffices to prove the first asymptotic in \eqref{Bound_tau_y_jj02} since $\tau_u \leq \tau_u^y$ for any $y \in \bb S_+^{d-1}$. 
By the definition of $\tau_u$, it holds that 
$\bb P \big( \tau_u \leq  \upzeta_u  \big) =  \bb P ( |V_{\upzeta_u}| > u ).$
For any $\beta \in (0, \rho)$, there exists $s \in I_{\mu}^{\circ}$ such that $\beta = \frac{1}{\Lambda'(s)}$. 
Since the function $\Lambda'$ is increasing on $I_{\mu}^{\circ}$, 
for any vanishing perturbation $l \geq 0$, there exists $l_s \geq 0$ satisfying $l_s \to 0$ as $u \to \infty$
such that 
\begin{align}\label{beta-l-Lambda}
\beta - l = \frac{1}{ \Lambda'(s + l_s)}. 
\end{align}
Applying Lemma \ref{Lem_Yn_u} with $n = \upzeta_u$ and $s = s + l_s$,  
there exist constants $\ee > 0$ and $c>0$ (independent of $\ee$) such that 
uniformly in $l_s \to 0$, $u >0$ and $n \geq 1$, 
\begin{align}\label{Ine_Vupzeta_ll}
 \bb P ( |V_{\upzeta_u}| > u )    
  \leq  c \,  \upzeta_u^{1 + s + l_s + \ee}  \exp \left\{ \upzeta_u  \left[ \ee \Lambda'(s+l_s) + \ee^2 \Lambda''(s+l_s) \right] \right\}  
  \kappa(s+l_s)^{\upzeta_u} u^{-(s + l_s + \ee)}. 
\end{align}
Following the proof of (3.8) and (3.9) in \cite{BCDZ16}, one can use \eqref{Def_I_beta_s} and \eqref{beta-l-Lambda} to derive 
\begin{align}\label{Pf_Iden_Ibeta}
 \kappa(s + l_s)^{(\beta - l) \log u} u^{-(s + l_s)} 
 = u^{ -I(\beta -l) }, 
 \qquad 
 \kappa(s+l_s)^{\upzeta_u}  u^{-(s + l_s)} \leq u^{ -I(\beta - l) }  \kappa(s+l_s)^{- a \log \log u}. 
\end{align}
We choose $\ee = \ee (u) = \frac{\log \log u}{2 \log u}$, 
so that $u^{-\ee} = \frac{1}{\sqrt{\log u}}$. 
In view of \eqref{Ine_Vupzeta_ll} and \eqref{Pf_Iden_Ibeta}, to prove the first asymptotic in \eqref{Bound_tau_y_jj02},  
it suffices to show that, as $u \to \infty$, uniformly in $l_s \to 0$, 
\begin{align}\label{Pf_tau_o_1}
 \upzeta_u^{1 + s + l_s + \ee}  \kappa(s+l_s)^{- a \log \log u}  
   \exp \left\{ \upzeta_u \left[ \ee \Lambda'(s+l_s) + \ee^2 \Lambda''(s+l_s) \right] \right\} 
 = o(1). 
\end{align}
Since $\upzeta_u \leq (\beta - l) \log u = \frac{ \log u }{\Lambda'(s + l_s)}$ (by \eqref{beta-l-Lambda}) and $\Lambda'(s+l_s) > 0$, 
 as $u \to \infty$, uniformly in $l \to 0$ and $l_s \to 0$, 
 it holds that $\exp \left\{ \upzeta_u  \left[ \ee \Lambda'(s+l_s) + \ee^2 \Lambda''(s+l_s) \right] \right\} 
 \leq c \sqrt{\log u}.$ 
Since $\kappa(s+l_s)^{- a \log \log u} = (\log u)^{- a \Lambda(s + l_s)}$ 
and $\upzeta_u^{1 + s + l_s + \ee}  \leq c' (\log u)^{1 + s + l_s}$ by using $(\log u)^{\ee} \leq c$, 
we have $\upzeta_u^{1 + s + l_s + \ee}  \kappa(s+l_s)^{- a \log \log u} 
\leq c (\log u)^{1 + s + l_s - a \Lambda(s + l_s)}.$ 
Therefore, we obtain \eqref{Pf_tau_o_1} by taking $a > \frac{3 + 2(s+l_s)}{2\Lambda(s+l_s)}$
and using the fact that $\Lambda(s + l_s) > 0$. 
By Taylor's formula,  for any $\eta>0$, it holds uniformly in $l_s \geq 0$ with $l_s \to 0$ that
$(1 + \eta) \frac{3 + 2s}{2\Lambda(s)} > \frac{3 + 2(s+l_s)}{2\Lambda(s+l_s)}$. 
Since $\eta>0$ can be arbitrary small, we conclude the proof of \eqref{Bound_tau_y_jj02}
by taking $a > \frac{3 + 2s}{2\Lambda(s)}$. 
\end{proof}

As another application of Lemma \ref{Lem_Yn_u}, 
we obtain the following bounds for $\tau_u$ and $\tau_u^y$, 
which will also be used to establish Theorem \ref{Thm_CLT_Perpe} 
on the conditioned central limit theorems for $\tau_u$ and $\tau_u^y$.

\begin{lemma}\label{Lem_tau_Low_b}
Let $\alpha$ and $\rho$ be defined by \eqref{Def_alpha} and \eqref{Def_rho}, respectively. 
Assume \ref{Condi_ExpMom}, \ref{Condi_AP} and
  that there exists $\ee \in (0,1)$ such that $\kappa(\alpha + t) < \infty$ and $\kappa_Q(\alpha + t) < \infty$ for all $|t| \leq \ee$. 
Then, for any $\delta >0$, there exists a constant $c>0$ such that
for any $b > \rho (1 + \alpha + \delta) + \rho^2 \sigma_{\alpha}^2$, $u > e$ and $y \in \bb S_+^{d-1}$, 
\begin{align}\label{Bound-bb-tau}
\bb P \left( \tau_u^y \leq \rho \log u - b \sqrt{ (\log u) (\log \log u) }  \right) \leq 
\bb P \left( \tau_u \leq \rho \log u - b \sqrt{ (\log u) (\log \log u) }  \right) 
\leq  c (\log u)^{-\delta} u^{-\alpha}. 
\end{align} 
\end{lemma}

\begin{proof}
Since $\tau_u \leq \tau_u^y$ for any $y \in \bb S_+^{d-1}$, 
it suffices to prove the second inequality in \eqref{Bound-bb-tau}. 
Setting $\upzeta_u = \floor{ \rho \log u - b \sqrt{ (\log u) (\log \log u) } }$
and applying Lemma \ref{Lem_Yn_u} with $s = \alpha$ and $n = \upzeta_u$, we get 
\begin{align}\label{pf-V-upzeta-u-lem}
\bb P \left( \tau_u \leq \rho \log u - b \sqrt{ (\log u) (\log \log u) }  \right) = 
\bb P ( |V_{\upzeta_u}| > u ) 
 \leq  c  \upzeta_u^{1 + \alpha + \ee}  
     \exp \left\{ \upzeta_u  \left[ \ee \rho^{-1} + \ee^2 \sigma_{\alpha}^2 \right] \right\}  
     u^{-(\alpha + \ee)}, 
\end{align}
where we used the fact that 
$\kappa(\alpha) = 1$, $\Lambda'(\alpha) = \rho^{-1}$ and $\Lambda''(\alpha) = \sigma_{\alpha}^2$. 
We choose $\ee = \ee (u) = \sqrt{ \frac{\log \log u}{\log u} }.$ 
Then, using $\upzeta_u \leq \rho \log u$ and $(\log u)^{\ee} \leq c$, we get $\upzeta_u^{1 + \alpha + \ee} \leq c (\log u)^{1+\alpha}$
and 
\begin{align*}
\exp \big\{ \upzeta_u  \left[ \ee \rho^{-1} + \ee^2 \sigma_{\alpha}^2 \right] \big\}
& \leq  \exp \big\{ \big( \rho \log u - b \sqrt{ (\log u) (\log \log u) } \big)  \left[ \ee \rho^{-1} + \ee^2 \sigma_{\alpha}^2 \right] \big\}  \notag\\
& = u^{\ee}  \exp \left\{ - b \rho^{-1} \log\log u + \rho \sigma_{\alpha}^2 \log\log u
     - b \sigma_{\alpha}^2  \frac{(\log\log u)^{3/2}}{(\log u)^{1/2}}   \right\}    \notag\\
& \leq  u^{\ee}  \exp \left\{ - b \rho^{-1} \log\log u + \rho \sigma_{\alpha}^2 \log\log u  \right\}   \notag\\
& =  u^{\ee} (\log u)^{ - b \rho^{-1} + \rho \sigma_{\alpha}^2  }. 
\end{align*} 
Hence $\upzeta_u^{1 + \alpha + \ee}  
  \exp \big\{ \upzeta_u  \left[ \ee \rho^{-1} + \ee^2 \sigma_{\alpha}^2 \right] \big\} u^{- \ee}
  \leq c (\log u)^{ 1 + \alpha - b \rho^{-1} + \rho \sigma_{\alpha}^2  },$ 
which, together with \eqref{pf-V-upzeta-u-lem}, concludes the proof of \eqref{Bound-bb-tau}
 by taking $b > \rho (1 + \alpha + \delta) + \rho^2 \sigma_{\alpha}^2$. 
\end{proof}

Next we investigate the right tail of the hitting time of the process $V-V_n$ to the level $u$. 
To this aim, we denote 
\begin{align}\label{Def_bar_Vn}
\bar{V}_n = V - V_n = \sum_{j = n}^{\infty} (M_1 \ldots M_j) Q_{j+1},   
\qquad    \bar{\tau}_u = \sup \left\{ n \geq 1: |\bar{V}_n| >u  \right\}
\end{align}
and for $y \in \bb S_+^{d-1}$, 
\begin{align*}
\bar{\tau}_u^y = \sup \left\{ n \geq 1:  \langle y, \bar{V}_n \rangle  >u  \right\}.  
\end{align*}
Since $\langle y, \bar{V}_n \rangle \leq |\bar{V}_n|$, we have $\bar{\tau}_u^y \leq \bar{\tau}_u$ for any $y \in \bb S_+^{d-1}$. 
As an analog of Lemma \ref{Lem_tau_Low_b}, 
the following result gives an estimate of the right tail of the exceeding time $\bar{\tau}_u$ of $\bar{V}_n$ to the level $u$.

\begin{lemma}\label{Lem_tau_Low_c}
Let $\alpha$ and $\rho$ be defined by \eqref{Def_alpha} and \eqref{Def_rho}, respectively. 
Assume \ref{Condi_ExpMom}, \ref{Condi_AP} and that there exists $\ee \in (0,1)$ such that 
$\kappa(\alpha + t) < \infty$ and $\kappa_Q(\alpha + t) < \infty$ for all $|t| \leq \ee$. 
Then, for any $\delta >0$, there exists a constant $c > 0$ 
such that for any $b > \max\{ \rho, 8 (2 \alpha + 1 + \delta) \rho^2 \sigma_{\alpha}^2 \}$, $u > e$ and $y \in \bb S_+^{d-1}$, 
\begin{align}\label{bound-bar-tau-u}
\bb P \left( \bar{\tau}_u^y  \geq  \rho \log u + b \sqrt{ (\log u) (\log \log u) }  \right)  \leq 
\bb P \left( \bar{\tau}_u  \geq  \rho \log u + b \sqrt{ (\log u) (\log \log u) }  \right) 
\leq  c (\log u)^{-\delta} u^{-\alpha}. 
\end{align}
\end{lemma}

\begin{proof}
Since the proof follows from that of \cite[Lemma 4.2]{BCDZ16},
we only outline the main differences. 
As $\bar{\tau}_u^y \leq \bar{\tau}_u$ for any $y \in \bb S_+^{d-1}$, it suffices to prove the second inequality in \eqref{bound-bar-tau-u}.
Proceeding in the same way as in the proof of (4.9) and (4.10) in \cite{BCDZ16},
we set $\upzeta_u = \floor{ \rho \log u + b \sqrt{ (\log u) (\log \log u) } }$ and use 
the independence of $\Pi_j$ and $Q_{j+1}$, and Lemma \ref{Lem_kappa}
to get that for small $\ee > 0$,  
\begin{align*}
\bb P \left( \bar{\tau}_u  \geq  \rho \log u + b \sqrt{ (\log u) (\log \log u) }  \right) = 
\bb P \left( |\bar{V}_{\upzeta_u} | > u \right) 
 \leq   c  u^{-\alpha} \sum_{j = 1}^{\infty} j^{2 \alpha} A_j(u), 
\end{align*}
where 
$A_j(u) = u^{\ee} 
  \exp \Big\{ (\upzeta_u + j -1) \big( - \frac{\ee}{\rho} +  \ee^2 \sigma_{\alpha}^2 \big) \Big\}$. 
Using the definition of $\upzeta_u$ and 
choosing $\ee =  \frac{  b \sqrt{ (\log u) (\log \log u) } + j-1  }{ 2 \rho \sigma_{\alpha}^2 (\upzeta_u + j -1) }$, 
we get 
\begin{align*}
A_j(u)  & \leq  \exp \Big\{ - \frac{\ee}{\rho}  \big[b \sqrt{ (\log u) (\log \log u) } + j-1 \big] 
        + \ee^2 \sigma_{\alpha}^2  (\upzeta_u + j -1)   \Big\}  \notag\\
& \leq   \exp \bigg\{ -  \frac{ \big[ b \sqrt{ (\log u) (\log \log u) } + j-1 \big]^2 }
                        { 4 \rho^2 \sigma_{\alpha}^2 (\upzeta_u + j -1) } \bigg\} =: B_j(u).  
\end{align*}
Since $B_j(u) \leq 
\exp \left\{ - \frac{b^2 \log \log u }{ 4 \rho^2 \sigma_{\alpha}^2 (\rho + b)  }  \right\}
\leq \exp \left\{ - \frac{b \log \log u}{ 8 \rho^2 \sigma_{\alpha}^2 }  \right\}$ for $1 \leq j \leq \upzeta_u$ and $b \geq \rho$, 
we get 
\begin{align*}
\sum_{j = 1}^{ \upzeta_u }  j^{2 \alpha} B_j(u) 
 \leq   \upzeta_u^{2 \alpha + 1}  
   \exp \bigg\{ - \frac{b \log \log u}{ 8 \rho^2 \sigma_{\alpha}^2 }  \bigg\}
   \leq  c (\log u)^{2 \alpha + 1 - \frac{b}{ 8\rho^2 \sigma_{\alpha}^2 }}.
\end{align*}
Since for $\upzeta_u \leq j-1$ and $b \geq \rho$, 
\begin{align*}
B_j(u) \leq  \exp \left\{ - \frac{  b \sqrt{ \log u }(j-1) }
     { 2\rho^2 \sigma_{\alpha}^2 (\upzeta_u + j -1) }
       -  \frac{ (j-1)^2 }{ 4\rho^2 \sigma_{\alpha}^2 (\upzeta_u + j -1) } \right\}
   \leq \exp \left\{ - \frac{ b \sqrt{\log u} }{ 4 \rho^2 \sigma_{\alpha}^2 }  -  \frac{ j-1}{ 8 \rho^2 \sigma_{\alpha}^2 } \right\}, 
\end{align*}
we get $\sum_{j = \upzeta_u + 1}^{ \infty } j^{2 \alpha} B_j(u) 
\leq c  \exp \Big\{ - \frac{ b \sqrt{\log u} }{ 4 \rho^2 \sigma_{\alpha}^2 }  \Big\}$. 
Combining the above estimates, we obtain 
$\bb P \left( |\bar{V}_{\upzeta_u} | > u \right)  
\leq  c   (\log u)^{2 \alpha + 1 - \frac{ b}{ 8\rho^2 \sigma_{\alpha}^2 }}  u^{-\alpha}.$
Taking $b > \max\{ \rho, 8 (2 \alpha + 1 + \delta) \rho^2 \sigma_{\alpha}^2 \}$,  
we conclude the proof of  \eqref{bound-bar-tau-u}.  
\end{proof}

Now we are equipped to establish the following stronger version of  Theorem \ref{Thm_LLN_Perpe} 
on the conditioned weak laws of large numbers for the first passage times $\tau_u$ and $\tau_u^y$.

\begin{theorem}\label{Thm_LLN_Perpe-stronger}
Let $\rho$ be defined by \eqref{Def_rho}. 
Assume \ref{Condi_ExpMom}, \ref{Condi_AP} and \ref{Condi_NonArith}. 
Then there exists a constant $b>0$ such that
\begin{align}\label{thm-LLN-tau-u}
\lim_{u \to \infty} \bb P 
\left(  \left. \left| \frac{\tau_u}{ \log u } - \rho \right| 
    >  b \frac{ \sqrt{\log \log u} }{ \sqrt{\log u} }  \, \right|  \tau_u < \infty \right)
= 0, 
\end{align}
and uniformly in $y \in \bb S_+^{d-1}$, 
\begin{align}\label{thm-LLN-tau-u-y}
\lim_{u \to \infty} \bb P 
\left( \left. \left| \frac{\tau_u^y}{ \log u } - \rho \right| >  b \frac{ \sqrt{\log \log u} }{ \sqrt{\log u} }   \,  \right|  \tau_u^y < \infty \right)
= 0. 
\end{align}
\end{theorem}

\begin{proof}[Proof of Theorems \ref{Thm_LLN_Perpe} and \ref{Thm_LLN_Perpe-stronger}]
As mentioned above, it suffices to establish Theorem \ref{Thm_LLN_Perpe-stronger} since Theorem \ref{Thm_LLN_Perpe}
is a consequence of Theorem \ref{Thm_LLN_Perpe-stronger}. 
We first prove \eqref{thm-LLN-tau-u}. 
On the one hand, from \eqref{Bound-bb-tau} of Lemma \ref{Lem_tau_Low_b} and \eqref{Exit_tail_V_n-intro}, we derive that 
\begin{align}\label{LLN_Upper}
\lim_{u \to \infty}
\bb P \left( \left.  \frac{\tau_u}{ \log u } -  \rho 
           < - b \frac{ \sqrt{\log \log u} }{ \sqrt{\log u} }  \right| \tau_u < \infty  \right) = 0.
\end{align}
On the other hand, setting $\upzeta_u = \floor{ \rho \log u + b \sqrt{ (\log u) (\log \log u) } }$
and recalling that $\bar{V}_n$ is defined by \eqref{Def_bar_Vn},  
from \eqref{bound-bar-tau-u} of Lemma \ref{Lem_tau_Low_c}, we get that, as $u \to \infty$, 
\begin{align}\label{tail-bar-V-upzeta}
\bb P \left( |\bar{V}_{\upzeta_u}| > u \right) 
= \bb P \left( \bar{\tau}_u  \geq  \rho \log u + b \sqrt{ (\log u) (\log \log u) }  \right) 
= o \left( u^{-\alpha} \right). 
\end{align}
Since $V = V_{\upzeta_u} + \bar{V}_{\upzeta_u}$, we claim that the tail decay of $V$ is determined by that of $V_{\upzeta_u}$. 
Indeed, since both $V_{\upzeta_u}$ and $\bar{V}_{\upzeta_u}$ belong to $\bb R_+^d$, it holds that
$\bb P ( |V| > u ) \geq \bb P ( | V_{\upzeta_u} | > u ).$
By the monotonicity of $\bar{V}_{n}$ in $n$ and \eqref{tail-bar-V-upzeta}, 
there exists $a(u) \downarrow 0$ such that as $u \to \infty$, 
\begin{align*}
\bb P \left( |\bar{V}_{\upzeta_u}| > a(u)  u \right) 
\leq  \bb P \left( |\bar{V}_{ \upzeta_{a(u) u} }| > a(u)  u \right) 
= o \left( (a(u) u)^{-\alpha} \right) = o \left( u^{-\alpha} \right). 
\end{align*}
It follows that 
\begin{align*}
\bb P ( |V| > u ) 
& =  \bb P \left( |V| > u, | \bar{V}_{\upzeta_u} | \leq a(u) u \right)
   + \bb P \left( |V| > u, | \bar{V}_{\upzeta_u} | > a(u) u \right)  \nonumber\\
& \leq \bb P \left( | V_{\upzeta_u} | > (1 - a(u))u \right) + \bb P \left( | \bar{V}_{\upzeta_u} | > a(u) u \right)  \notag\\
&   =   \bb P \left( | V_{\upzeta_u} | > (1 - a(u))u \right) +  o \left( u^{-\alpha} \right). 
\end{align*}
Hence
\begin{align}\label{Asymptotic-V-V-upzeta}
\bb P (|V| > u)  \sim  \bb P (|V_{\upzeta_u}| > u)   \quad  \mbox{as} \  u \to \infty. 
\end{align}
Since $\{ |V_{\upzeta_u}| > u \} \subseteq  \{ |V| > u \}$ and $\{ \tau_u < \infty \} = \{ |V| > u \}$, we obtain 
\begin{align*}
 \lim_{u \to \infty}
    \bb P \left(  \left.  \frac{\tau_u}{ \log u } -  \rho 
        \leq  b \frac{ \sqrt{\log \log u} }{ \sqrt{\log u} }  \right| \tau_u < \infty  \right)  
=  \lim_{u \to \infty} \frac{ \bb P (|V_{\upzeta_u}| > u) }{ \bb P (|V| > u) } = 1. 
\end{align*}
Combining this with \eqref{LLN_Upper}, we conclude the proof of \eqref{thm-LLN-tau-u}. 

The proof of \eqref{thm-LLN-tau-u-y} can be done in the same way by using \eqref{Exit_tail_V_n} and Lemmas \ref{Lem_tau_Low_b} and \ref{Lem_tau_Low_c}. 
\end{proof}

\section{Proof of Theorem \ref{Thm_CLT_Perpe}}

The goal of this section is to establish Theorem \ref{Thm_CLT_Perpe} 
on the conditioned central limit theorems for the first passage times $\tau_u$ and $\tau_u^y$. 
In the proof, we adapt the powerful method developed in \cite{BCDZ16} for the univariate case, 
incorporating two new elements which turn out to be crucial in our multivariate setting. 
The first one is concerned with the precise asymptotic behavior of the joint distribution of $(\frac{V}{|V|}, V)$, as detailed in Lemma \ref{Lem_V_weak_Uni_bb} below. 
Its proof builds on Lemma \ref{Lem_V_weak_Uni} (see below), 
and Lemma \ref{Lem_CM_Weak_Conv} which refines Kesten's asymptotic \eqref{Exit_tail_V_n-intro} and is established in \cite{CM18}. 
The second key element involves the following Bahadur-Rao-Petrov type large deviation asymptotics for  
the vector norm $|\Pi_n x|$ and the coefficients $\langle y, \Pi_n x \rangle$, 
both of which are also central to establishing Theorem \ref{Thm_LD_Perpe_Petrov}. 

\begin{theorem} \label{PetrovThm}  
Let $s\in I_{\mu}^{\circ}$ and $q=\Lambda'(s)$. 
Let $(l_n)_{n \geq 1}$ be any sequence of positive numbers such that $\lim_{n \to \infty} l_n = 0$. 
If \ref{Condi_ExpMom}, \ref{Condi_AP} and \ref{Condi_NonArith} hold, 
then we have, as $n \to \infty$, uniformly in $|l| \leq l_n$ and $x \in \bb S_+^{d-1}$, 
\begin{align}\label{LDPet_VectorNorm}
\mathbb{P} \big( \log |\Pi_n x|  \geq n(q+l) \big) = 
   \frac{r_s(x)}{ \nu_s(r_s) } 
   \frac{ \exp \left( -n \Lambda^*(q+l) \right)  }{s\sigma_s\sqrt{2\pi n}} ( 1 + o(1)).
\end{align}
If \ref{Condi_ExpMom}, \ref{Condi_NonArith} and \ref{Condi_Kes_Weak} hold, 
then we have, as $n \to \infty$, uniformly in $|l| \leq l_n$ and $x, y \in \bb S_+^{d-1}$, 
\begin{align} \label{LDPet_ScalarProduct}
 \mathbb{P} \big( \log \langle y, \Pi_n x \rangle \geq n(q+l) \big)  
 =  \frac{r_s(x) r^*_s(y)}{ \nu_s(r_s) } 
\frac{ \exp \left( -n   \Lambda^*(q+l) \right) } {s \sigma_{s}\sqrt{2\pi n}} ( 1 + o(1) ).  
\end{align}
\end{theorem}

The asymptotics \eqref{LDPet_VectorNorm} and \eqref{LDPet_ScalarProduct} have been recently established in \cite{XGL19a}
and \cite{XGL22}, respectively.  
We also mention that, for products of invertible random matrices, 
precise large deviation asymptotics for the coefficients $\langle y, \Pi_n x \rangle$
has also been obtained recently in \cite{XGL23}. 
We shall see below that the perturbation $l$ on $q$ in \eqref{LDPet_VectorNorm} and \eqref{LDPet_ScalarProduct}
plays a very important role in the proof of Theorem \ref{Thm_CLT_Perpe}. 

By \cite[Lemma 4.1]{XGL19a}, 
the rate function $\Lambda^*(q+l)$ admits the following expansion: for $q=\Lambda'(s)$ and $l$ in a small neighborhood
of $0$, 
\begin{align} \label{Def Jsl}
\Lambda^*(q+l)
   = \Lambda^{*}(q) + sl + \frac{l^2}{2 \sigma_s^2} - \frac{l^3}{\sigma_s^3} \xi_s\left(\frac{l}{\sigma_s}\right), 
\end{align}
where $\xi_s$ is defined by \eqref{def-xi-s-t}.

Now we proceed to establish \eqref{CLT-tau-u} of Theorem \ref{Thm_CLT_Perpe}. 
By Lemma \ref{Lem_tau_Low_b}, 
for any $\delta >0$, there exists a constant $c>0$ such that
for any $b > \rho (1 + \alpha + \delta) + \rho^2 \sigma_{\alpha}^2$ and $u > e$,  
\begin{align}\label{Pf_CLT_001}
\bb P \left( \tau_u \leq \rho \log u - b \sqrt{ (\log u) (\log \log u) }  \right) 
\leq c (\log u)^{-\delta} u^{-\alpha}. 
\end{align}
For fixed $b > \rho (1 + \alpha + \delta) + \rho^2 \sigma_{\alpha}^2$ and $t \in \bb R$,  we set
\begin{align}\label{Def_L_rho_u}
L_{\rho} (u) = b \sqrt{ (\log \log u) / \log u }, 
\quad 
\upzeta_u = \floor{ (\rho - L_{\rho} (u))  \log u },  
\quad
\eta_{u,t} = \floor{ \rho \log u + t \rho^{3/2} \sigma_{\alpha} \sqrt{\log u} }. 
\end{align}
We denote
\begin{align}\label{Def_Yu}
\mathcal{Y}_{u,t}  = Q_{\upzeta_u + 1} + M_{\upzeta_u + 1} Q_{\upzeta_u + 2} + \ldots 
    + (M_{\upzeta_u + 1} \ldots M_{\eta_{u,t} - 1}) Q_{\eta_{u,t}}, 
    \quad  \mathcal{X}_{u,t}  = \frac{ \mathcal{Y}_{u,t} }{ |\mathcal{Y}_{u,t}| }. 
\end{align}
Since $V_{\eta_{u,t}} = V_{\upzeta_u} + \Pi_{\upzeta_u} \mathcal{Y}_{u,t}$,   
by using \eqref{Pf_CLT_001} and repeating the arguments in the proof of \eqref{Asymptotic-V-V-upzeta},
one can verify that the contribution of the first part $V_{\upzeta_u}$ can be neglected. Namely, as $u \to \infty$, 
\begin{align}\label{pf-CLT-equivalence-tau-u}
\bb P \big( \tau_u \leq  \eta_{u,t} ) =  \bb P \left( | V_{\eta_{u,t}} | > u \right)
    = \bb P ( |\Pi_{\upzeta_u} \mathcal{Y}_{u,t}| > u ) ( 1 + o(1) ),  
\end{align}
and uniformly in $y \in \bb S_+^{d-1}$, 
\begin{align}\label{pf-CLT-equivalence-tau-u-y}
\bb P \big( \tau_u^y \leq  \eta_{u,t} ) =  \bb P \left( \langle y, V_{\eta_{u,t}} \rangle  > u \right)
    = \bb P ( \langle y, \Pi_{\upzeta_u} \mathcal{Y}_{u,t} \rangle > u ) ( 1 + o(1) ).  
\end{align}
We denote by $F_{u,t}$ the probability distribution function of the vector 
$\mathcal{Y}_{u,t}$ on $\bb R^d_+$, i.e. $F_{u,t}(dv) = \bb P (\mathcal{Y}_{u,t} \in dv)$, 
and write $\tilde{v} = v/|v|$ for the direction of $v \in \bb R^d_+ \setminus \{0\}$.
Then, by \eqref{Def_Yu} and the independence of $\Pi_{\upzeta_u}$ and $\mathcal{Y}_{u,t}$, 
\begin{align}\label{Pf_CLT_Id_a}
\bb P \left( | \Pi_{\upzeta_u} \mathcal{Y}_{u,t} | > u \right) 
& = \bb P \left( \log | \Pi_{\upzeta_u} \mathcal{X}_{u,t} | > \log u - \log | \mathcal{Y}_{u,t} |  \right)  \nonumber\\
& = \int_{\bb R^d_+}  \bb P( \log |\Pi_{\upzeta_u} \tilde{v}| > \log u - \log |v|) F_{u,t}(dv)
\end{align}
and, similarly, for any $y \in \bb S_+^{d-1}$,
\begin{align}\label{Pf_CLT_Id_a-y}
\bb P \left( \langle y, \Pi_{\upzeta_u} \mathcal{Y}_{u,t} \rangle > u \right) 
 = \int_{\bb R^d_+}  \bb P( \log \langle y, \Pi_{\upzeta_u} \tilde{v} \rangle > \log u - \log |v|) F_{u,t}(dv). 
\end{align}
For any $v \in \bb R_+^d \setminus \{0\}$, 
the inequality $\log |\Pi_{\upzeta_u} \tilde{v}| > \log u - \log |v|$ holds if and only if 
\begin{align}\label{Pf_Def_l_uv_002}
\frac{ \log |\Pi_{\upzeta_u} \tilde{v}| }{ \upzeta_u } > \frac{ \log u - \log |v| }{ \upzeta_u }
 = :  \frac{1}{\rho} + l_{u,v},
\end{align}
where, by \eqref{Def_L_rho_u}, 
\begin{align}\label{Property_l_uv}
l_{u,v} 
 = \frac{1}{\upzeta_u}  \left( \log u - \frac{1}{\rho} \upzeta_u - \log |v|  \right)
 = \frac{1}{\upzeta_u}  \left( \frac{L_{\rho} (u)}{\rho} \log u  - \log |v| + \delta_u \right), 
\end{align}
with $|\delta_u| \leq \frac{1}{\rho}$. 
Similarly, for $v \in \bb R_+^d \setminus \{0\}$ and $y \in \bb S_+^{d-1}$, 
the inequality $\log \langle y, \Pi_{\upzeta_u} \tilde{v} \rangle > \log u - \log |v|$ holds if and only if 
$\frac{ \log \langle y, \Pi_{\upzeta_u} \tilde{v} \rangle }{ \upzeta_u } > \frac{ \log u - \log |v| }{ \upzeta_u }
 =   \frac{1}{\rho} + l_{u,v}$. 
For brevity, we denote for fixed $t \in \bb R$ and $\Delta >0$, 
\begin{align}\label{Def_D_u_I_u}  
D(u) = \frac{L_{\rho} (u)}{\rho} \log u  + \delta_u,
\quad
I_{u,t} = \left[ 0, D(u) + \sigma_{\alpha} (t + \Delta) \sqrt{\upzeta_u} \right],
\quad
A_{u,t} = \left\{ \log | \mathcal{Y}_{u,t} | \in  I_{u,t}  \right\}, 
\end{align}
and by $A_{u,t}^c$ its complement.
Consider the following decompositions: 
\begin{align}
\bb P (| \Pi_{\upzeta_u} \mathcal{Y}_{u,t} | > u) 
& = \bb P (| \Pi_{\upzeta_u} \mathcal{Y}_{u,t} | > u, \, A_{u,t} )  + \bb P (| \Pi_{\upzeta_u} \mathcal{Y}_{u,t} | > u, \,  A_{u,t}^c),
 \label{Pf_CLT_Decom_P} \\
\bb P ( \langle y, \Pi_{\upzeta_u} \mathcal{Y}_{u,t} \rangle > u) 
&  = \bb P ( \langle y, \Pi_{\upzeta_u} \mathcal{Y}_{u,t} \rangle > u, \, A_{u,t} )  
    + \bb P ( \langle y, \Pi_{\upzeta_u} \mathcal{Y}_{u,t} \rangle > u, \,  A_{u,t}^c).  \label{Pf_CLT_Decom_P-y}
\end{align}
We will show that the second terms in \eqref{Pf_CLT_Decom_P} and \eqref{Pf_CLT_Decom_P-y} are negligible compared with the first terms. 
Our first objective is to obtain precise asymptotics 
for the first terms as $u \to \infty$, 
by applying Theorem \ref{PetrovThm} on the precise large deviation asymptotics
for products of random matrices. 

\begin{lemma}
Let $\rho$ be defined by \eqref{Def_rho}. 
If \ref{Condi_ExpMom}, \ref{Condi_AP} and \ref{Condi_NonArith} hold, 
then, for any $t \in \bb R$, as $u \to \infty$, 
\begin{align}\label{Pf_Upzeta_Yu_Au}
 \bb P (| \Pi_{\upzeta_u} \mathcal{Y}_{u,t} | > u, \,  A_{u,t} )    
= \frac{ u^{- \alpha} }{\sqrt{2 \pi}}   
  \int_{ \bb D_{u,t} }  e^{- \frac{1}{2} (\log |v|)^2} H_{u,t}(dv) ( 1 + o(1)), 
\end{align}
where 
\begin{align}\label{def-Du-Ju}
\bb D_{u,t} = \left\{ v \in \bb R_+^d:  \log |v| \in J_{u,t}  \right\}  \quad 
\mbox{with} \  J_{u,t} =  \left[ - \frac{ D(u) }{ \sigma_{\alpha} \sqrt{\upzeta_u} }, \,  t + \Delta \right]
\end{align}
and 
for any Borel set $B \subset \bb R_+^d$, 
with $T_u^{-1}(v) = e^{ D(u) + \sigma_{\alpha} \sqrt{\upzeta_u} |v| } \frac{v}{|v|}$ for any $v \in \bb R^d \setminus \{0\}$, 
\begin{align}\label{def-Hu-B}
H_{u,t}(B)  =  \frac{1}{\alpha \sigma_{\alpha}  \nu_{\alpha}(r_{\alpha}) \sqrt{\upzeta_u}}
     \int_{T_u^{-1} (B)}  r_{\alpha} \left( \frac{v}{|v|} \right)  e^{\alpha \log |v|}  F_{u,t}(dv). 
\end{align}
If \ref{Condi_ExpMom}, \ref{Condi_NonArith} and \ref{Condi_Kes_Weak} hold, 
then for any $y \in \bb S_+^{d-1}$, as $u \to \infty$, 
\begin{align}\label{Pf_Upzeta_Yu_Au-y}
 \bb P ( \langle y, \Pi_{\upzeta_u} \mathcal{Y}_{u,t} \rangle > u, \,  A_{u,t} )    
= r_{\alpha}^*(y) \frac{ u^{- \alpha} }{\sqrt{2 \pi}}   
  \int_{ \bb D_{u,t} }  e^{- \frac{1}{2} (\log |v|)^2} H_{u,t}(dv) ( 1 + o(1)). 
\end{align}
\end{lemma}

\begin{proof}
We first prove \eqref{Pf_Upzeta_Yu_Au} 
by applying the Bahadur-Rao-Petrov type large deviation asymptotics for the vector norm $|\Pi_{\upzeta_u} \tilde{v}|$,  
cf.\ \eqref{LDPet_VectorNorm}. 
The method of proof is similar to that of (4.26) in \cite{BCDZ16}. 
Recall that $\rho = 1/ \Lambda'(\alpha)$ with $\alpha >0$ (cf.\ \eqref{Def_rho}) 
and $\upzeta_u = \floor{ \rho \log u} (1 + o(1))$ as $u \to \infty$ (cf.\  \eqref{Def_L_rho_u}). 
For brevity, we denote
\begin{align}\label{def-bb-Bu}
\bb B_{u,t} = \left\{ v \in \bb R_+^d:  \log |v| \in I_{u,t}  \right\},
\end{align}
where $I_{u,t}$ is given in \eqref{Def_D_u_I_u}. 
As $u \to \infty$, it holds uniformly in $\log |v| \in  I_{u,t}$ that
$l_{u,v} \to 0$ and $\upzeta_u l_{u,v}^3 \to 0$. 
From \eqref{Pf_CLT_Id_a}, the definition of $A_{u,t}$ and $\bb B_{u,t}$, \eqref{Property_l_uv} and \eqref{Def_L_rho_u},
we get that with $\tilde{v} = v/|v|$, 
\begin{align*}
 \bb P (| \Pi_{\upzeta_u} \mathcal{Y}_{u,t} | > u, \,  A_{u,t} )    
 =  \int_{\bb B_{u,t}} \bb P \big( \log |\Pi_{\upzeta_u} \tilde{v}| >  \upzeta_u \left( \Lambda'(\alpha)  + l_{u,v} \right) \big) F_{u,t}(dv). 
\end{align*}
By \eqref{LDPet_VectorNorm}, 
\eqref{Def Jsl} and the fact that $\upzeta_u l_{u,v}^3 \to 0$,    
it follows that as $u \to \infty$, 
uniformly in $\log |v| \in I_{u,t}$, 
\begin{align*}
 \bb P ( | \Pi_{\upzeta_u} \mathcal{Y}_{u,t} | > u, \,  A_{u,t} )    
 = \frac{1}{\alpha \sigma_{\alpha}  \nu_{\alpha}(r_{\alpha}) \sqrt{2 \pi \upzeta_u}}    \int_{ \bb B_{u,t} } r_{\alpha}(\tilde{v})  
   e^{- \upzeta_u \Lambda^*(\Lambda'(\alpha)) - \alpha \upzeta_u  l_{u,v} - \upzeta_u l_{u,v}^2 / (2 \sigma_{\alpha}^2)} 
      F_{u,t}(dv) ( 1 + o(1)).  
\end{align*}
Since $\Lambda(\alpha) = 0$ (cf. \eqref{Def_alpha}) and $\Lambda'(\alpha) = 1/\rho$, 
we have $\Lambda^*(\Lambda'(\alpha)) = \alpha \Lambda'(\alpha) - \Lambda(\alpha) = \alpha / \rho$.
From the first equality in \eqref{Property_l_uv}, we see that  
$- \upzeta_u  \frac{\alpha}{\rho} - \alpha \upzeta_u  l_{u,v} = \alpha \log |v| - \alpha \log u$.
By \eqref{Def_D_u_I_u} and the second equality in \eqref{Property_l_uv}, 
it holds that $l_{u,v} = \frac{1}{\upzeta_u}  (D(u) - \log |v|)$. 
Hence
\begin{align*}
e^{- \upzeta_u \Lambda^*(\Lambda'(\alpha)) - \alpha \upzeta_u  l_{u,v} - \upzeta_u l_{u,v}^2 / (2 \sigma_{\alpha}^2)}
 = u^{- \alpha}  e^{\alpha \log |v|}    e^{ - \frac{1}{2} |T_u(v)|^2 },   
\end{align*}
where $T_u: \bb R^d \setminus \{0\} \to \bb R^d$ is defined by 
\begin{align}\label{Def_Tu_aa}
T_u(v) = \frac{ \log |v| - D(u) }{ \sigma_{\alpha} \sqrt{\upzeta_u} }  \frac{v}{|v|}.  
\end{align}
Therefore, 
\begin{align*}
& \bb P ( | \Pi_{\upzeta_u} \mathcal{Y}_{u,t} | > u, \,  A_{u,t} )   
 =  \frac{ u^{- \alpha} }{\alpha \sigma_{\alpha}  \nu_{\alpha}(r_{\alpha}) \sqrt{2 \pi \upzeta_u}}   
  \int_{\bb B_u} r_{\alpha}(\tilde{v})  
   e^{\alpha \log |v|}    e^{- \frac{1}{2} |T_u(v)|^2 }
      F_{u,t}(dv) ( 1 + o(1)). 
\end{align*}
Under the mapping $T_u$, the domain $\bb B_{u,t}$ (cf.\ \eqref{def-bb-Bu}) is transformed into $\bb D_{u,t}$ given by \eqref{def-Du-Ju}. 
Note that, under $T_u$, 
the direction $\tilde{v}$ of each vector $v \in \bb D_{u,t}$ is the same as that in $\bb B_{u,t}$. 
Denote $G_{u,t}(B) = F_{u,t}(T_u^{-1} B)$ for any Borel set $B \subset \bb R_+^d$. 
By a change of variable, it follows that as $u \to \infty$, 
\begin{align*}
 \bb P ( | \Pi_{\upzeta_u} \mathcal{Y}_{u,t} | > u, \,  A_{u,t} )    
& =  \frac{ u^{- \alpha} }{\alpha \sigma_{\alpha}  \nu_{\alpha}(r_{\alpha}) \sqrt{2 \pi \upzeta_u}}   
  \int_{ \bb D_{u,t} } r_{\alpha}(\tilde{v})  
   e^{\alpha \log | T_u^{-1} v |}    e^{- \frac{1}{2} |v|^2}
      G_{u,t}(dv) ( 1 + o(1))   \nonumber\\
& = \frac{ u^{- \alpha} }{\sqrt{2 \pi}}   
  \int_{\bb D_{u,t}}  e^{- \frac{1}{2} |v|^2} H_{u,t}(dv) ( 1 + o(1)), 
\end{align*}
where in the last line we used the fact that for any Borel set $B \subset \bb R_+^d$, 
\begin{align*}
H_{u,t}(B) 
 = \frac{1}{\alpha \sigma_{\alpha}  \nu_{\alpha}(r_{\alpha}) \sqrt{\upzeta_u}}
    \int_B  r_{\alpha}(\tilde{v})  e^{\alpha \log | T_u^{-1} v |}  G_{u,t}(dv).  
\end{align*}
This ends the proof of \eqref{Pf_Upzeta_Yu_Au}. 
Using \eqref{LDPet_ScalarProduct} instead of \eqref{LDPet_VectorNorm},
the proof of \eqref{Pf_Upzeta_Yu_Au-y} can be done in the same way. 
\end{proof}

In order to study the asymptotic behavior of the integral on the right hand side of \eqref{Pf_Upzeta_Yu_Au}, 
it is central to characterize the measure $H_{u,t}$ on $\bb D_{u,t}$. 
To this aim, we shall first establish a weak convergence result for the measure $H_{u,t}$ as $u \to \infty$. 
For $- \infty < a_1 < a_2 < \infty$, denote
\begin{align*}
\bb D_{a_1, a_2} = \left\{ v \in \bb R_+^d:  \log |v| \in [a_1, a_2]  \right\}.
\end{align*} 
We start by giving an upper bound for $H_{u,t} (\bb D_{a_1, a_2})$, 
which will be used in the proof of Proposition \ref{Prop_CLT_aaa}. 
Note that it is important that the constant $c$ does not depend on $u$, $a_1$ and $a_2$.

\begin{lemma}\label{Lem_Huannuals_Bound}
Assume \ref{Condi_ExpMom}, \ref{Condi_AP} and \ref{Condi_NonArith}. 
Then there exists a constant $c> 0$ such that for any $- \infty < a_1 < a_2 < \infty$, $t \in \bb R$ and $u \geq 0$, 
\begin{align*}
H_{u,t} (\bb D_{a_1, a_2}) \leq c (a_2 - a_1) + \frac{c}{\sqrt{\upzeta_u}}.   
\end{align*}
\end{lemma}

\begin{proof}
We denote by $F$ the probability distribution function of the vector 
$V$ on $\bb R^d_+$, i.e. $F(dv) = \bb P (V \in dv)$.  
In view of \eqref{def-Hu-B}, 
since the eigenfunction $r_{\alpha}$ is uniformly bounded on $\bb S_+^{d-1}$,  
there exists a constant $c = c(\alpha) >0$ such that for any $t \in \bb R$ and $u \geq 0$, 
\begin{align}\label{Pf_H_u_D}
H_{u,t} (\bb D_{a_1, a_2}) 
 \leq   \frac{c}{ \sqrt{\upzeta_u} }  \int_{T_u^{-1} (\bb D_{a_1, a_2})}  e^{\alpha \log |v|} F_{u,t}(dv)  
 \leq   \frac{c}{ \sqrt{\upzeta_u} }  \int_{T_u^{-1} (\bb D_{a_1, a_2})}  e^{\alpha \log |v|} F(dv), 
\end{align}
where in the last inequality we used the fact that $F_{u,t}(B) \leq F(B)$ for any Borel measurable set $B \subseteq \bb R^d_+$. 
By the definition of $T_u$ (cf. \eqref{Def_Tu_aa}), we have 
\begin{align*}
T_u^{-1} (\bb D_{a_1, a_2}) 
= \left\{ v \in \bb R_+^d:  
     \log |v| \in \left[ D(u) + a_1 \sigma_{\alpha} \sqrt{\upzeta_u}, D(u) + a_2 \sigma_{\alpha} \sqrt{\upzeta_u} \right]  \right\}. 
\end{align*}
Since $F$ is the distribution function of the random variable $V$ on $\bb R_+^d$,  
it follows that 
\begin{align}\label{Pf_Identity_V}
\int_{T_u^{-1} (\bb D_{a_1, a_2})}  e^{\alpha \log |v|} F(dv) 
= \bb E  \bigg( |V|^{\alpha} 
    \mathds 1_{ \big\{ |V| \in \left[ e^{D(u) + a_1 \sigma_{\alpha} \sqrt{\upzeta_u} }, \, 
                        e^{D(u) + a_2 \sigma_{\alpha} \sqrt{\upzeta_u} } \right]  \big\} }  \bigg). 
\end{align}
Denote by $\bar{F}$ the distribution function of the random variable $|V|$ on $\bb R_+$. 
Using Fubini's theorem, we get 
\begin{align*}
 \int_{T_u^{-1} (\bb D_{a_1, a_2})}  e^{\alpha \log |v|} F(dv) 
& = \int_{0}^{\infty}  z^{\alpha}  \mathds 1_{ \big\{ z \in \big[ e^{D(u) + a_1 \sigma_{\alpha} \sqrt{\upzeta_u} }, \, 
                        e^{D(u) + a_2 \sigma_{\alpha} \sqrt{\upzeta_u} } \big]  \big\} }  \bar{F}(dz)   \notag\\
& = \int_{0}^{\infty}  \left(  \int_0^{\infty} \alpha w^{\alpha - 1} \mathds 1_{\{ z \geq w \}} dw  \right)
      \mathds 1_{ \big\{ z \in \big[ e^{D(u) + a_1 \sigma_{\alpha} \sqrt{\upzeta_u} }, \, 
                        e^{D(u) + a_2 \sigma_{\alpha} \sqrt{\upzeta_u} } \big]  \big\} }  \bar{F}(dz)   \notag\\
& = \int_0^{\infty} \alpha w^{\alpha - 1}  \bb P \left( |V| \geq w, \, 
               |V| \in \left[ e^{D(u) + a_1 \sigma_{\alpha} \sqrt{\upzeta_u} }, \, 
                        e^{D(u) + a_2 \sigma_{\alpha} \sqrt{\upzeta_u} } \right] \right)  dw   \notag\\
& = E_1 + E_2,  
\end{align*}
where 
\begin{align*}
 E_1  & = \int_0^{ e^{D(u) + a_1 \sigma_{\alpha} \sqrt{\upzeta_u} } }  \alpha w^{\alpha - 1}  dw
     \,  \bb P \left(  |V| \in \left[ e^{D(u) + a_1 \sigma_{\alpha} \sqrt{\upzeta_u} }, \, 
                        e^{D(u) + a_2 \sigma_{\alpha} \sqrt{\upzeta_u} } \right] \right),  \notag\\
 E_2  & =  \int_{ e^{D(u) + a_1 \sigma_{\alpha} \sqrt{\upzeta_u} } }^{ e^{D(u) + a_2 \sigma_{\alpha} \sqrt{\upzeta_u} } }  
           \alpha w^{\alpha - 1} 
     \,  \bb P \left(  |V| \in \left[ w, \, 
                        e^{D(u) + a_2 \sigma_{\alpha} \sqrt{\upzeta_u} } \right] \right)  dw. 
\end{align*}
By \eqref{Exit_tail_V_n-intro}, there exists a constant $c >0$ such that
$\bb P (|V| \geq w) \leq c w^{- \alpha}$ for any $w \in \bb R_+$, 
so
\begin{align}\label{Pf_Bound_E1}
E_1 \leq  e^{ \alpha (D(u) + a_1 \sigma_{\alpha} \sqrt{\upzeta_u}) } 
  \bb P \left(  |V| \geq   e^{D(u) + a_1 \sigma_{\alpha} \sqrt{\upzeta_u} }   \right)
  \leq  c.
\end{align}
Similarly, 
\begin{align}\label{Pf_Bound_E2}
E_2   \leq   \int_{ e^{D(u) + a_1 \sigma_{\alpha} \sqrt{\upzeta_u} } }^{ e^{D(u) + a_2 \sigma_{\alpha} \sqrt{\upzeta_u} } }  
           \alpha  w^{\alpha - 1} 
     \,  \bb P \left(  |V| \geq  w  \right)  dw    
  \leq  c   
    \int_{ e^{D(u) + a_1 \sigma_{\alpha} \sqrt{\upzeta_u} } }^{ e^{D(u) + a_2 \sigma_{\alpha} \sqrt{\upzeta_u} } }
    \frac{1}{w}  dw    
 \leq c (a_2 - a_1) \sqrt{\upzeta_u}. 
\end{align}
Substituting \eqref{Pf_Bound_E1} and \eqref{Pf_Bound_E2} into \eqref{Pf_H_u_D} 
concludes the proof of Lemma \ref{Lem_Huannuals_Bound}. 
\end{proof}

To obtain an exact asymptotic for $H_{u,t} (\bb D_{a_1, a_2})$, 
we need to study the asymptotic behavior of the joint law of $(\mathcal{X}_{u,t} \in \cdot, | \mathcal{Y}_{u,t} | \geq  e^t)$ 
with $t$ belonging to an interval which moves to infinity as $u \to \infty$. 
To this aim, 
we shall prove a stronger version of  Lemma \ref{Lem_CM_Weak_Conv},  
which plays a very important role in establishing Theorem \ref{Thm_CLT_Perpe}. 
We first apply \eqref{Exit_tail_V_n-intro} and Lemma \ref{Lem_tau_Low_c} 
to prove the following result which will be used in the proof of Lemma \ref{Lem_V_weak_Uni_bb}. 

\begin{lemma}\label{Lem_V_weak_Uni}
Assume \ref{Condi_ExpMom}, \ref{Condi_AP} and \ref{Condi_NonArith}. 
Let $\mathscr C >0$ be the constant given by \eqref{Exit_tail_V_n-intro}. 
Then, for any $t \in \bb R$, we have
\begin{align}\label{Uni_Convergen_V}
\lim_{u \to \infty}  \sup_{w \in J_u}
\left| e^{\alpha w} \bb P  \left( | \mathcal{Y}_{u,t} | \geq  e^w  \right) - \mathscr{C}  \right|
= 0,
\end{align}
where $\mathcal{Y}_{u,t}$ is defined by \eqref{Def_Yu} and 
\begin{align}\label{def-Ju}
J_u = \left[ D(u) + a_1 \sigma_{\alpha} \sqrt{\upzeta_u}, \,  D(u) + a_2 \sigma_{\alpha} \sqrt{\upzeta_u} \right]. 
\end{align}
\end{lemma}

\begin{proof}
Set $m_{u,t} =  \eta_{u,t} - \upzeta_u$, 
where $\upzeta_u$ and $\eta_{u,t}$ are given in \eqref{Def_L_rho_u}.  
Then 
\begin{align}\label{Def_mu_001}
m_{u,t} 
 = \floor{ \rho \log u + t \rho^{3/2} \sigma_{\alpha} \sqrt{\log u} }
    - \floor{  (\rho - L_{\rho} (u)) \log u }   
 = L_{\rho} (u) \log u + t \rho^{3/2} \sigma_{\alpha} \sqrt{\log u} + \delta_u,
\end{align}
where $|\delta_u| \leq 1$ and $L_{\rho} (u)$ is defined by \eqref{Def_L_rho_u}.  
Notice that for any $\ee>0$ and $t \in \bb R$, 
it holds that $\bb P  \left( |\bar{V}_{m_{u,t}}| > \ee e^w \right) = \bb P \left( \bar{\tau}_{\ee e^w}  \geq  m_{u,t}  \right)$, 
where $\bar{V}_n$ and $\bar{\tau}_{u}$ are defined in \eqref{Def_bar_Vn}.
We shall first prove the following inequality:
for any $\delta > 0$, 
there exists a constant $c > 0$ such that for any $\ee>0$ and $w \in J_u$ with $w > - \log \ee + 1$, 
\begin{align}\label{Inequa_bar_Vn}
e^{\alpha w} \bb P \left( | \bar{V}_{m_{u,t}} | > \ee e^w  \right) 
\leq  c \ee^{-\alpha} \left( w + \log \ee \right)^{-\delta}. 
\end{align}
Using Lemma \ref{Lem_tau_Low_c} with $u = \ee e^w$, we get
that for any $\delta >0$, there exist constants $b, c > 0$ such that for any $\ee>0$ and $w > - \log \ee + 1$, 
\begin{align*}
e^{\alpha w} \bb P \left( \bar{\tau}_{\ee e^w}  
  \geq  \rho \left( w + \log \ee \right) + b \sqrt{ \left( w + \log \ee \right) \log \left( w + \log \ee \right)} \right)
\leq  c \ee^{-\alpha} \left( w + \log \ee \right)^{-\delta}. 
\end{align*}
Hence, to prove \eqref{Inequa_bar_Vn}, it suffices to show that
\begin{align*}
m_{u,t} \geq  \rho \left( w + \log \ee \right) + b \sqrt{ \left( w + \log \ee \right) \log \left( w + \log \ee \right)}
   \quad  \mbox{with} \  w =  D(u) + a_2 \sigma_{\alpha} \sqrt{\upzeta_u}. 
\end{align*}
This is clear by \eqref{Def_mu_001},  \eqref{Def_L_rho_u}, \eqref{Def_D_u_I_u} and the fact that $t >0$. 
Therefore, \eqref{Inequa_bar_Vn} holds.

We next apply \eqref{Inequa_bar_Vn} to prove \eqref{Uni_Convergen_V}. 
For the upper bound, from \eqref{Exit_tail_V_n-intro} we have
\begin{align*}
\lim_{u \to \infty}  \sup_{w \in J_u}
\left| e^{\alpha w} \bb P  \left( |V| \geq  e^w  \right) - \mathscr{C}  \right|
= 0. 
\end{align*}
Since $\bb P  \left( | \mathcal{Y}_{u,t} | \geq  e^w  \right) \leq \bb P  \left( |V| \geq  e^w  \right),$
it follows that for any $t \in \bb R$, 
\begin{align}\label{limsup-Yu-C}
\limsup_{u \to \infty}  \sup_{w \in J_u}
  e^{\alpha w} \bb P  \left( | \mathcal{Y}_{u,t} | \geq  e^w  \right) \leq \mathscr{C}. 
\end{align}
For the lower bound, since for any $\ee >0$, 
\begin{align*}
\left\{  |V| \geq (1 + \ee) e^w   \right\}  \subseteq
\left\{  |\bar{V}_{ m_{u,t} }| \geq  \ee e^w   \right\} \bigcup
\left\{  |V - \bar{V}_{ m_{u,t} }| \geq  e^w   \right\}, 
\end{align*}
using the fact that $\mathcal{Y}_{u,t}$ has the same distribution as that of $V - \bar{V}_{ m_{u,t} }$, we get
\begin{align*}
\bb P  \left( | \mathcal{Y}_{u,t} | \geq  e^w  \right)
= \bb P  \left( |V - \bar{V}_{ m_{u,t} } | \geq  e^w  \right)
\geq  \bb P  \left( |V| \geq (1 + \ee) e^w  \right) -  \bb P  \left( |\bar{V}_{ m_{u,t} }| \geq  \ee e^w  \right). 
\end{align*}
By \eqref{Exit_tail_V_n-intro}, we have that for any $\ee>0$, 
\begin{align*}
\liminf_{u \to \infty}  \inf_{w \in J_u}
 e^{\alpha w} \bb P  \left( |V| \geq  (1 + \ee)  e^w  \right) \geq  (1+\ee)^{-\alpha} \mathscr{C}. 
\end{align*}
Using \eqref{Inequa_bar_Vn},  for any $\ee>0$,  
we get $\lim_{u \to \infty}  \sup_{w \in J_u}  e^{\alpha w}  \bb P  \left( |\bar{V}_{ m_{u,t} }| \geq  \ee e^w  \right) = 0.$
Therefore, for any $\ee>0$, 
\begin{align}\label{liminf-Yu-C}
\liminf_{u \to \infty}  \inf_{w \in J_u}
  e^{\alpha w} \bb P  \left( | \mathcal{Y}_{u,t} | \geq  e^w  \right) \geq  (1+\ee)^{-\alpha} \mathscr{C}. 
\end{align}
Since $\ee>0$ can be arbitrary small, 
combining \eqref{limsup-Yu-C} and \eqref{liminf-Yu-C} concludes the proof of \eqref{Uni_Convergen_V}. 
\end{proof}

Using Lemmas \ref{Lem_V_weak_Uni} and \ref{Lem_CM_Weak_Conv}, 
now we further study the joint law of $(\mathcal{X}_{u,t}, \mathcal{Y}_{u,t})$. 

\begin{lemma}\label{Lem_V_weak_Uni_bb}
Assume \ref{Condi_ExpMom}, \ref{Condi_AP} and \ref{Condi_NonArith}. 
Let $\mathscr C >0$ be the constant given by \eqref{Exit_tail_V_n-intro}. 
Then, for any $t \in \bb R$, we have
\begin{align}\label{Uni_Convergen_V_bb}
\lim_{u \to \infty} 
\sup_{w \in J_u}
\left| e^{\alpha w}
\bb E  \Big[ r_{\alpha} \left( \mathcal{X}_{u,t} \right)  \mathds 1_{ \{ | \mathcal{Y}_{u,t} | \geq  e^w  \} }  \Big]
- \mathscr{C} \nu_{\alpha}(r_{\alpha})  \right|
= 0,
\end{align}
where $\mathcal{Y}_{u,t}$ is defined by \eqref{Def_Yu}, $\mathcal{X}_{u,t} = \mathcal{Y}_{u,t} / | \mathcal{Y}_{u,t} |$ and 
$J_u$ is defined by \eqref{def-Ju}. 
\end{lemma}

\begin{proof}
Denote $\mathcal{Y} = V$ and $\mathcal{X} = \mathcal{Y} / |\mathcal{Y}|$. 
From \eqref{Weak_Conver_V} we see that 
\begin{align*}
\lim_{u \to \infty} 
\sup_{w \in J_u}
\left| e^{\alpha w}
\bb E  \left[ r_{\alpha} \left( \mathcal{X} \right)  \mathds 1_{ \{ |\mathcal{Y}| \geq  e^w  \} }  \right]
- \mathscr{C} \nu_{\alpha}(r_{\alpha})  \right|
= 0. 
\end{align*}
Therefore, in order to establish \eqref{Uni_Convergen_V_bb}, it suffices to prove that for any $t \in \bb R$, 
\begin{align}\label{Pf_Equivalent_Form_Yu}
\lim_{u \to \infty} 
\sup_{w \in J_u}  e^{\alpha w}
\left| \bb E  \left[ r_{\alpha} \left( \mathcal{X} \right)  \mathds 1_{ \{ |\mathcal{Y}| \geq  e^w  \} }  \right]
  - \bb E  \left[ r_{\alpha} \left( \mathcal{X}_{u,t} \right)  \mathds 1_{ \{ | \mathcal{Y}_{u,t} | \geq  e^w  \} }  \right]  
   \right|
= 0. 
\end{align}
By the triangular inequality, we have
\begin{align*}
e^{\alpha w} \left| 
\bb E  \left[ r_{\alpha} \left( \mathcal{X} \right)  \mathds 1_{ \{ |\mathcal{Y}| \geq  e^w  \} }  \right]
 - \bb E  \left[ r_{\alpha} \left( \mathcal{X}_{u,t} \right)  \mathds 1_{ \{ | \mathcal{Y}_{u,t} | \geq  e^w  \} }  \right] 
   \right|
\leq  I_1(t,u,w) + I_2(t,u,w),
\end{align*}
where 
\begin{align*}
&  I_1(t,u,w) = e^{\alpha w} \left| \bb E  \left[ r_{\alpha} \left( \mathcal{X} \right)  
   \left( \mathds 1_{ \{ |\mathcal{Y}| \geq  e^w  \} } 
       - \mathds 1_{ \{ | \mathcal{Y}_{u,t} | \geq  e^w  \} } \right)  \right]  \right|,  \nonumber\\
&  I_2(t,u,w) = e^{\alpha w}
    \left| \bb E  \left[ \left( r_{\alpha} \left( \mathcal{X}_{u,t} \right) - r_{\alpha} \left( \mathcal{X} \right) \right)
         \mathds 1_{ \{ | \mathcal{Y}_{u,t} | \geq  e^w  \} }  \right]   \right|. 
\end{align*}

\textit{Bound of $I_1(t,u,w)$.}
Since the eigenfunction $r_{\alpha}$ is positive and bounded on $\bb S_+^{d-1}$, 
and that $\{ |\mathcal{Y}_{u,t} | \geq  e^w \} \subset \{ |\mathcal{Y}| \geq  e^w \}$,
there exists a constant $c>0$ such that
\begin{align*}
I_1(t,u,w) \leq c  e^{\alpha w} 
   \left[ \bb P \left( |\mathcal{Y}| \geq  e^w \right) - \bb P \left( | \mathcal{Y}_{u,t} | \geq  e^w \right)  \right]. 
\end{align*}
From \eqref{Exit_tail_V_n-intro}, we get 
\begin{align*}
\lim_{u \to \infty}  \sup_{w \in J_u}
\left| e^{\alpha w} \bb P  \left( |\mathcal{Y}| \geq  e^w  \right) - \mathscr{C}  \right|
= 0. 
\end{align*}
By Lemma \ref{Lem_V_weak_Uni}, we have that for any $t \in \bb R$, 
\begin{align*}
\lim_{u \to \infty}  \sup_{w \in J_u}
\left| e^{\alpha w} \bb P  \left( | \mathcal{Y}_{u,t} | \geq  e^w  \right) - \mathscr{C}  \right|
= 0. 
\end{align*}
Therefore, for any $t \in \bb R$, 
\begin{align}\label{Bound_I1_tu}
\lim_{u \to \infty}  \sup_{w \in J_u}  I_1(t,u,w) = 0. 
\end{align}

\textit{Bound of $I_2(t,u,w)$.}
Since the eigenfunction $r_{\alpha}$ is $\ee$-H\"{o}lder continuous on $\bb S_+^{d-1}$, 
there exists a constant $c_{\ee} >0$ such that 
\begin{align*}
I_2(t,u,w) \leq   c_{\ee} e^{\alpha w} \bb E  \left( \left| \mathcal{X}_{u,t} - \mathcal{X}  \right|^{\ee}
         \mathds 1_{ \{ |\mathcal{Y}_{u,t}| \geq  e^w  \} }  \right). 
\end{align*}
Notice that 
\begin{align*}
\left| \mathcal{X}_{u,t} - \mathcal{X}  \right|
& = \left| \frac{\mathcal{Y}_{u,t} }{|\mathcal{Y}_{u,t} |} -  \frac{\mathcal{Y}}{|\mathcal{Y}|} \right|
  = \left| \frac{\mathcal{Y}_{u,t} |\mathcal{Y} - \mathcal{Y} |\mathcal{Y}| 
        + \mathcal{Y} |\mathcal{Y}| - \mathcal{Y} |\mathcal{Y}_{u,t}| |}{| \mathcal{Y}_{u,t} | |\mathcal{Y}|}  \right|  \nonumber\\
& \leq  \frac{| \mathcal{Y}_{u,t} - \mathcal{Y} |}{|\mathcal{Y}_{u,t}|} 
      + \frac{| |\mathcal{Y}| - |\mathcal{Y}_{u,t}| |}{|\mathcal{Y}_{u,t} |}   
  \leq 2 \frac{| \mathcal{Y}_{u,t} - \mathcal{Y} |}{|\mathcal{Y}_{u,t}|}. 
\end{align*}
It follows that  
\begin{align*}
I_2(t,u,w) & \leq   c_{\ee} e^{\alpha w}  \bb E  \left( \frac{| \mathcal{Y}_{u,t} - \mathcal{Y} |^{\ee}}{| \mathcal{Y}_{u,t} |^{\ee}}
         \mathds 1_{ \{ | \mathcal{Y}_{u,t} | \geq  e^w  \} }  \right)  \nonumber\\
& \leq  c_{\ee}  e^{\alpha w - \ee w}  \bb E  \left( | \mathcal{Y}_{u,t} - \mathcal{Y} |^{\ee}
         \mathds 1_{ \{ | \mathcal{Y}_{u,t} | \geq  e^w  \} }  \right)    \nonumber\\
& = I_{21}(t,u,w) + I_{22}(t,u,w),
\end{align*}
where, with $a \in (0,1)$ a fixed constant, 
\begin{align*}
& I_{21}(t,u,w) = c_{\ee}  e^{\alpha w - \ee w}  \bb E  \left( | \mathcal{Y}_{u,t} - \mathcal{Y} |^{\ee}
         \mathds 1_{ \{ | \mathcal{Y}_{u,t} - \mathcal{Y} | \leq  e^{aw}  \} }
         \mathds 1_{ \{ | \mathcal{Y}_{u,t} | \geq  e^w  \} }  \right),   \nonumber\\
& I_{22}(t,u,w) = c_{\ee}  e^{\alpha w - \ee w}  \bb E  \left( | \mathcal{Y}_{u,t} - \mathcal{Y} |^{\ee}
         \mathds 1_{ \{ | \mathcal{Y}_{u,t} - \mathcal{Y} | >  e^{aw}  \} }
         \mathds 1_{ \{ | \mathcal{Y}_{u,t} | \geq  e^w  \} }  \right). 
\end{align*}

\textit{Bound of $I_{21}(t,u,w)$.}
By lemma \ref{Lem_V_weak_Uni}, we get that, as $u$ sufficiently large,
uniformly in $w \in J_u$, 
\begin{align*}
I_{21}(t,u,w) \leq  c_{\ee}  e^{\alpha w - \ee w + a \ee w} \bb P \left( | \mathcal{Y}_{u,t} | \geq  e^w  \right) 
\leq e^{- (1 -a) \ee w}. 
\end{align*}
Since $a \in (0,1)$, this gives that for any $t \in \bb R$, 
\begin{align}\label{Bound_I21_tu}
\lim_{u \to \infty}  \sup_{w \in J_u}  I_{21}(t,u,w) = 0. 
\end{align}

\textit{Bound of $I_{22}(t,u,w)$.}
By H\"{o}lder's inequality, we have that for any $p, p' > 1$ with $1/p + 1/p' = 1$, 
\begin{align*}
I_{22}(t,u,w) 
\leq  c_{\ee}  e^{\alpha w - \ee w}  
        \left[ \bb E  \left( | \mathcal{Y}_{u,t} - \mathcal{Y} |^{p \ee}
         \mathds 1_{ \{ | \mathcal{Y}_{u,t} - \mathcal{Y} | >  e^{aw}  \} }   \right)
         \right]^{1/p}
        \left[ \bb P \left(  | \mathcal{Y}_{u,t} | \geq  e^w   \right)   \right]^{1/p'}. 
\end{align*}
Notice that the distribution of the random variable $| \mathcal{Y}_{u,t} - \mathcal{Y} |$ coincides with that of $|\bar{V}_{m_{u,t}}|$,
where $\bar{V}_{m_{u,t}}$ is the same as in \eqref{Inequa_bar_Vn}. 
Hence, by Lemma \ref{Lem_V_weak_Uni} and the fact that $1 - 1/p' = 1/p$, it follows that uniformly in $w \in J_u$,  
\begin{align*}
I_{22}(t,u,w)  \leq  c_{\ee}  e^{\frac{\alpha}{p} w - \ee w}  
        \left[ \bb E  \left( | \bar{V}_{m_{u,t}} |^{p \ee}
         \mathds 1_{ \{ | \bar{V}_{m_{u,t}} | >  e^{aw}  \} }   \right)
         \right]^{1/p}.  
\end{align*}
Using integration by parts and the inequality \eqref{Inequa_bar_Vn}, 
we deduce that there exist constants $c > 0$ and $\delta >1$ such that uniformly in $w \in J_u$,  
\begin{align*}
 \bb E  \left( | \bar{V}_{m_{u,t}} |^{p \ee}
         \mathds 1_{ \{ | \bar{V}_{m_{u,t}} | >  e^{aw}  \} }   \right)  
& = e^{p\ee aw} \bb P \left( | \bar{V}_{m_{u,t}} | >  e^{aw} \right)
  + \int_{e^{aw}}^{\infty} \bb P \left( | \bar{V}_{m_{u,t}} | >  x \right) x^{p \ee - 1} dx  \nonumber\\
& \leq  c a^{-\delta} w^{-\delta}  e^{p\ee aw - \alpha a w} 
   +  c  \int_{e^{aw}}^{\infty} \frac{1}{x^{\alpha - p \ee + 1} \log^{\delta} x}  dx  \nonumber\\
& = c a^{-\delta} w^{-\delta} + c  \int_{e^{aw}}^{\infty} \frac{1}{x \log^{\delta} x}  dx,
\end{align*}
where in the last line we chose $\ee \in (0,1)$ and $p > 1$ such that $p \ee = \alpha$. 
With this choice, it holds that $e^{\frac{\alpha}{p} w - \ee w} = 1$. 
Consequently, using the inequality $(a + b)^{1/p} \leq a^{1/p} + b^{1/p}$ for $a, b \geq 0$ and $p > 1$, 
we obtain that uniformly in $w \in J_u$,   
\begin{align*}
I_{22}(t,u,w)  \leq  c_{\ee, p, \delta, a} w^{-\frac{\delta}{p}}
   + c_{\ee, p} \left( \int_{e^{aw}}^{\infty} \frac{1}{x \log^{\delta} x}  dx \right)^{1/p}.   
\end{align*}
Since $\delta > 1$, this yields that 
\begin{align}\label{Bound_I22_tu}
\lim_{u \to \infty}  \sup_{w \in J_u}  I_{22}(t,u,w) = 0. 
\end{align}
Putting together \eqref{Bound_I1_tu}, \eqref{Bound_I21_tu} and \eqref{Bound_I22_tu}, 
we conclude the proof of \eqref{Pf_Equivalent_Form_Yu}. 
\end{proof}

Using Lemmas \ref{Lem_V_weak_Uni} and \ref{Lem_V_weak_Uni_bb}, now we can give an exact asymptotic for $H_{u,t} (\bb D_{a_1, a_2})$.
Recall that $\bb D_{a_1, a_2} = \{ v \in \bb R_+^d:  \log |v| \in [a_1, a_2] \}.$ 

\begin{lemma}\label{Lem_Weak_Conver_Hu}
Assume \ref{Condi_ExpMom}, \ref{Condi_AP} and \ref{Condi_NonArith}. 
Let $\mathscr C >0$ be the constant given by \eqref{Exit_tail_V_n-intro}.  
Let $- \infty < a_1 < a_2 < \infty$. 
Then, for any $t \in \bb R$,  we have 
\begin{align}\label{Limit_Hu_a}
\lim_{u \to \infty}  H_{u,t} (\bb D_{a_1, a_2}) =  \mathscr C (a_2 - a_1). 
\end{align}
More generally, for any $t \in \bb R$ 
and any continuous and compactly supported function $f$ on $\bb R$, 
we have
\begin{align}\label{Limit_Hu_b}
\lim_{u \to \infty} \int_{\bb D_{a_1, a_2}}  f(\log |v|) H_{u,t}(dv) = \mathscr C  \int_{a_1}^{a_2} f(s) ds.   
\end{align}
\end{lemma}

\begin{proof}
It suffices to prove \eqref{Limit_Hu_a} since \eqref{Limit_Hu_b} is a consequence of \eqref{Limit_Hu_a}. 
Recall that $\mathcal{Y}_{u,t}$ and $\mathcal{X}_{u,t}$ are defined in \eqref{Def_Yu}, 
and that $F_{u,t}$ is the probability distribution function of $\mathcal{Y}_{u,t}$ on $\bb R^d_+$. 
Similarly to \eqref{Pf_Identity_V}, we have
\begin{align*}
 H_{u,t}(\bb D_{a_1, a_2})  
 =  \frac{1}{\alpha \sigma_{\alpha}  \nu_{\alpha}(r_{\alpha}) \sqrt{\upzeta_u}}
     \int_{T_u^{-1} (\bb D_{a_1, a_2})}  r_{\alpha}(\tilde{v})  e^{\alpha \log |v|}  F_{u,t}(dv)  
 =  \frac{1}{\alpha \sigma_{\alpha}  \nu_{\alpha}(r_{\alpha}) \sqrt{\upzeta_u}} K, 
\end{align*}
where 
\begin{align*}
K = \bb E  \bigg[ |\mathcal{Y}_{u,t}|^{\alpha} 
      \mathds 1_{ \left\{ |\mathcal{Y}_{u,t}| \in \left[ e^{D(u) + a_1 \sigma_{\alpha} \sqrt{\upzeta_u} }, \, 
                        e^{D(u) + a_2 \sigma_{\alpha} \sqrt{\upzeta_u} } \right]  \right\} }  
                         r_{\alpha} \left( \mathcal{X}_{u,t} \right)  \bigg]. 
\end{align*}
For convenience, denote by $\tilde{F}_{u,t}$ the joint distribution of the couple $(|\mathcal{Y}_{u,t}|, \mathcal{X}_{u,t})$,  
namely, 
\begin{align*}
\tilde{F}_{u,t} (dz, dx) = \bb P \left(  |\mathcal{Y}_{u,t}| \in dz, \mathcal{X}_{u,t} \in dx \right). 
\end{align*}
Then it holds that 
\begin{align*}
K = \int_{\bb S_+^{d-1}} \int_0^{\infty} z^{\alpha} 
  \mathds 1_{ \left\{ z \in \left[ e^{D(u) + a_1 \sigma_{\alpha} \sqrt{\upzeta_u} }, \, 
                        e^{D(u) + a_2 \sigma_{\alpha} \sqrt{\upzeta_u} } \right]  \right\} }
      r_{\alpha}(x) \tilde{F}_{u,t} (dz, dx). 
\end{align*}
By Fubini's theorem, we get
\begin{align*}
K & = \int_{\bb S_+^{d-1}} \int_0^{\infty} 
     \left(  \int_0^{\infty} \alpha w^{\alpha - 1} \mathds 1_{\{ z \geq w \}} dw  \right)   
        \mathds 1_{ \left\{ z \in \left[ e^{D(u) + a_1 \sigma_{\alpha} \sqrt{\upzeta_u} }, \, 
                        e^{D(u) + a_2 \sigma_{\alpha} \sqrt{\upzeta_u} } \right]  \right\} }
        r_{\alpha}(x) \tilde{F}_{u,t} (dz, dx)   \nonumber\\
& = \alpha \int_0^{\infty}  w^{\alpha - 1} 
   \bigg[ \int_{\bb S_+^{d-1}} \int_0^{\infty} 
   \mathds 1_{\{ z \geq w \}}  \mathds 1_{ \left\{ z \in \left[ e^{D(u) + a_1 \sigma_{\alpha} \sqrt{\upzeta_u} }, \, 
                        e^{D(u) + a_2 \sigma_{\alpha} \sqrt{\upzeta_u} } \right]  \right\} }   
                           r_{\alpha}(x) \tilde{F}_{u,t} (dz, dx) \bigg] dw   \nonumber\\
& = \alpha \int_0^{\infty}  w^{\alpha - 1}
    \bb E  \left[ r_{\alpha} \left( \mathcal{X}_{u,t} \right)
          \mathds 1_{ \left\{ |\mathcal{Y}_{u,t}| \geq w, \,   |\mathcal{Y}_{u,t}| \in \left[ e^{D(u) + a_1 \sigma_{\alpha} \sqrt{\upzeta_u} }, \, 
                        e^{D(u) + a_2 \sigma_{\alpha} \sqrt{\upzeta_u} } \right]  \right\} }  
                           \right]  dw  \nonumber\\
& = K_1 + K_2,
\end{align*}
where 
\begin{align*}
& K_1 = \alpha \int_0^{ e^{D(u) + a_1 \sigma_{\alpha} \sqrt{\upzeta_u} } }  w^{\alpha - 1}  dw   
  \bb E  \left[ r_{\alpha} \left( \mathcal{X}_{u,t} \right)
          \mathds 1_{ \left\{ |\mathcal{Y}_{u,t}| \in \left[ e^{D(u) + a_1 \sigma_{\alpha} \sqrt{\upzeta_u} }, \, 
                        e^{D(u) + a_2 \sigma_{\alpha} \sqrt{\upzeta_u} } \right]  \right\} }  
                           \right],    \nonumber\\
& K_2 =  \alpha  \int_{ e^{D(u) + a_1 \sigma_{\alpha} \sqrt{\upzeta_u} } }^{ e^{D(u) + a_2 \sigma_{\alpha} \sqrt{\upzeta_u} } } 
      w^{\alpha - 1}
    \bb E  \left[ r_{\alpha} \left( \mathcal{X}_{u,t} \right)
          \mathds 1_{ \left\{ |\mathcal{Y}_{u,t}| \in \left[ w, \, 
                        e^{D(u) + a_2 \sigma_{\alpha} \sqrt{\upzeta_u} } \right]  \right\} }  
                           \right]  dw. 
\end{align*}
For $K_1$, since the eigenfunction $r_{\alpha}$ is bounded on $\bb S_+^{d-1}$,
using \eqref{Inequa_bar_Vn} we get 
\begin{align}\label{Pf_Bound_K1}
K_1 \leq c e^{ \alpha (D(u) + a_1 \sigma_{\alpha} \sqrt{\upzeta_u}) } 
  \bb P \left(  |\mathcal{Y}_{u,t}| \geq   e^{D(u) + a_1 \sigma_{\alpha} \sqrt{\upzeta_u} }   \right)
  \leq  C.
\end{align}
For $K_2$, we decompose it into two parts: $K_2 = K_{21} - K_{22}$, where
\begin{align*}
& K_{21} =  \alpha  \int_{ e^{D(u) + a_1 \sigma_{\alpha} \sqrt{\upzeta_u} } }^{ e^{D(u) + a_2 \sigma_{\alpha} \sqrt{\upzeta_u} } } 
      w^{\alpha - 1}
    \bb E  \left[ r_{\alpha} \left( \mathcal{X}_{u,t} \right)
          \mathds 1_{ \{ |\mathcal{Y}_{u,t}| \geq  w  \} }  
                          \right]  dw,    \nonumber\\
& K_{22} =  \alpha \int_{ e^{D(u) + a_1 \sigma_{\alpha} \sqrt{\upzeta_u} } }^{ e^{D(u) + a_2 \sigma_{\alpha} \sqrt{\upzeta_u} } } 
      w^{\alpha - 1}  dw
   \,  \bb E  \Big[ r_{\alpha} \left( \mathcal{X}_{u,t} \right)
          \mathds 1_{ \left\{ |\mathcal{Y}_{u,t}| \geq  e^{D(u) + a_2 \sigma_{\alpha} \sqrt{\upzeta_u} }  \right\} }  
                           \Big]. 
\end{align*}
Concerning $K_{21}$, by a change of variable $\log w = w'$, it holds that 
\begin{align*}
K_{21} = \alpha  \int_{ D(u) + a_1 \sigma_{\alpha} \sqrt{\upzeta_u}  }^{ D(u) + a_2 \sigma_{\alpha} \sqrt{\upzeta_u} } 
     e^{\alpha w}
    \bb E  \left[ r_{\alpha} \left( \mathcal{X}_{u,t} \right)
          \mathds 1_{ \left\{ |\mathcal{Y}_{u,t}| \geq  e^w  \right\} }  
                          \right]  dw.  
\end{align*}
Using Lemma \ref{Lem_V_weak_Uni_bb}, we get that, 
for any $t \in \bb R$, 
uniformly in $w \in J_u = [D(u) + a_1 \sigma_{\alpha} \sqrt{\upzeta_u}, D(u) + a_2 \sigma_{\alpha} \sqrt{\upzeta_u}]$,
\begin{align*}
\lim_{u \to \infty} 
 e^{\alpha w}
\bb E  \left[ r_{\alpha} \left( \mathcal{X}_{u,t} \right)
          \mathds 1_{ \left\{ |\mathcal{Y}_{u,t}| \geq  e^w  \right\} }  
                          \right]
 =  \mathscr{C} \nu_{\alpha}(r_{\alpha}). 
\end{align*}
Therefore, 
\begin{align}\label{Pf_Bound_K21}
\lim_{u \to \infty}
\frac{1}{\alpha \sigma_{\alpha}  \nu_{\alpha}(r_{\alpha}) \sqrt{\upzeta_u}} K_{21}
= \mathscr C (a_2 - a_1). 
\end{align}
Concerning $K_{22}$, using \eqref{Inequa_bar_Vn}, we get 
\begin{align}\label{Pf_Bound_K22}
K_{22} \leq  c e^{ \alpha (D(u) + a_2 \sigma_{\alpha} \sqrt{\upzeta_u}) } 
  \bb P \left(  |\mathcal{Y}_{u,t}| \geq   e^{D(u) + a_2 \sigma_{\alpha} \sqrt{\upzeta_u} }   \right)
  \leq  C.
\end{align}
Putting together \eqref{Pf_Bound_K1}, \eqref{Pf_Bound_K21} and \eqref{Pf_Bound_K22},
we conclude the proof of \eqref{Limit_Hu_a}. 
\end{proof}

Now we are equipped to give exact asymptotics for the first terms in \eqref{Pf_CLT_Decom_P} and \eqref{Pf_CLT_Decom_P-y}. 

\begin{proposition}\label{Prop_CLT_aaa}
Let $\rho$ be defined by \eqref{Def_rho}. 
If \ref{Condi_ExpMom}, \ref{Condi_AP} and \ref{Condi_NonArith} hold,  
then for any $t \in \bb R$, 
\begin{align}\label{Asym-Pi-upzeta-Yu-Au}
\lim_{u \to \infty} u^{\alpha} \, \bb P (|\Pi_{\upzeta_u} \mathcal{Y}_{u,t}| > u, A_{u,t}) = \mathscr C \Phi(t). 
\end{align}
If \ref{Condi_ExpMom}, \ref{Condi_NonArith} and \ref{Condi_Kes_Weak} hold, 
then for any $t \in \bb R$ and $y \in \bb S_+^{d-1}$, 
\begin{align}\label{Asym-Pi-upzeta-Yu-Au-y}
\lim_{u \to \infty} u^{\alpha} \, \bb P ( \langle y, \Pi_{\upzeta_u} \mathcal{Y}_{u,t} \rangle > u, \,  A_{u,t}) = \mathscr C  r_{\alpha}^*(y) \Phi(t). 
\end{align}
\end{proposition}

\begin{proof}
We first prove \eqref{Asym-Pi-upzeta-Yu-Au}. 
For $b > 1$, we denote
\begin{align*}
& \bb D_{u,b}^{(1)} = \left\{ v \in \bb R_+^d:  
    \log |v| \in  \left[ - \frac{ D(u) }{ \sigma_{\alpha} \sqrt{\upzeta_u} }, \, -b \right]  \right\}, \nonumber\\
& \bb D_{t,b}^{(2)} = \left\{ v \in \bb R_+^d:  
      \log |v| \in  \left(- b, t + \Delta \right)  \right\}. 
\end{align*}
Then $\bb D_{u,t} = \bb D_{u,b}^{(1)} \cup \bb D_{t,b}^{(2)}$
and hence from \eqref{Pf_Upzeta_Yu_Au} we have that as $u \to \infty$, 
\begin{align*}
 u^{\alpha} \bb P (|\Pi_{\upzeta_u} \mathcal{Y}_{u,t}| > u, \,  A_{u,t})   
& = \frac{1}{\sqrt{2 \pi}}   
  \int_{\bb D_{u,t}}  e^{- \frac{1}{2} (\log |v|)^2} H_{u,t}(dv) ( 1 + o(1))  \nonumber\\
& = \frac{1}{\sqrt{2 \pi}} \left(  \int_{\bb D_{u,b}^{(1)}}  e^{- \frac{1}{2} (\log |v|)^2} H_{u,t}(dv)
   +  \int_{\bb D_{t,b}^{(2)}}  e^{- \frac{1}{2} (\log |v|)^2} H_{u,t}(dv) \right) (1 + o(1)). 
\end{align*}
For the first term, using Lemma \ref{Lem_Huannuals_Bound}, 
we get that there exists a constant $c >0$ such that for any $t \in \bb R$ and $b > 1$, 
\begin{align*}  
\int_{\bb D_{u,b}^{(1)}}  e^{- \frac{1}{2} (\log |v|)^2} H_{u,t}(dv)
& \leq  c \int_{ - \frac{ D(u) }{ \sigma_{\alpha} \sqrt{\upzeta_u} } }^{-b} e^{-\frac{1}{2} y^2} dy 
    + \frac{c}{\sqrt{\upzeta_u}}   \nonumber\\
& \leq   c \int_{ -\infty }^{-b} e^{-\frac{1}{2} y^2} dy 
    + \frac{c}{\sqrt{\upzeta_u}}  
    \leq  c e^{- \frac{1}{2} b^2} + \frac{c}{\sqrt{\upzeta_u}},
\end{align*}
where we used the inequality $\int_{ -\infty }^{-b} e^{-\frac{1}{2} y^2} dy \leq e^{- \frac{1}{2} b^2}$.
Consequently,
\begin{align}\label{Pf_Au_Du1}
\lim_{b \to \infty}
\lim_{u \to \infty} \int_{\bb D_{u,b}^{(1)}}  e^{- \frac{1}{2} (\log |v|)^2} H_{u,t}(dv) = 0. 
\end{align}
For the second term, using Lemma \ref{Lem_Weak_Conver_Hu}, we obtain that for any $t \in \bb R$, 
\begin{align*}
\lim_{u \to \infty} \int_{\bb D_{t,b}^{(2)}}  e^{- \frac{1}{2} (\log |v|)^2} H_{u,t}(dv)
= \mathscr C  \int_{-b}^{t + \Delta} e^{-\frac{1}{2} y^2} dy, 
\end{align*}
so that 
\begin{align}\label{Pf_Au_Du2}
\lim_{\Delta \to 0} \lim_{b \to \infty}
\lim_{u \to \infty}  \int_{\bb D_{t,b}^{(2)}}  e^{- \frac{1}{2} (\log |v|)^2} H_{u,t}(dv)
= \mathscr C  \int_{-\infty}^{t} e^{-\frac{1}{2} y^2} dy. 
\end{align}
Combining \eqref{Pf_Au_Du1} and \eqref{Pf_Au_Du2}, we conclude the proof of \eqref{Asym-Pi-upzeta-Yu-Au}. 
Using \eqref{Pf_Upzeta_Yu_Au-y} instead of \eqref{Pf_Upzeta_Yu_Au}, the proof of \eqref{Asym-Pi-upzeta-Yu-Au-y} is similar. 
\end{proof}

Now we proceed to give estimates for the second terms in \eqref{Pf_CLT_Decom_P} and \eqref{Pf_CLT_Decom_P-y}. 

\begin{proposition}\label{Prop_CLT_bbb}
Assume \ref{Condi_ExpMom}, \ref{Condi_AP} and \ref{Condi_NonArith}. 
Let $\rho$ be defined by \eqref{Def_rho}. 
Then, for any $t \in \bb R$ and $y \in \bb S_+^{d-1}$, 
\begin{align}\label{Asyp-second-Pi-upzeta}
\lim_{u \to \infty} u^{\alpha} \bb P (|\Pi_{\upzeta_u} \mathcal{Y}_{u,t}| > u, A_{u,t}^c) = 0, 
\qquad
\lim_{u \to \infty} u^{\alpha} \bb P ( \langle y, \Pi_{\upzeta_u} \mathcal{Y}_{u,t} \rangle > u, A_{u,t}^c) = 0.  
\end{align}
\end{proposition}

\begin{proof}
The proof of this proposition goes along the same lines as step 4 of the proof of Theorem 2.2 in \cite{BCDZ16}, 
where $\mathcal G_u$ is replaced by $A_{u,t}$ in our notation, 
see pages 3719--3721. 
We also refer to the arXiv version of our article (see arXiv:2307.04985) for details of the proof.
\end{proof}

\begin{proof}[Proof of Theorem \ref{Thm_CLT_Perpe}]
Theorem \ref{Thm_CLT_Perpe} follows from \eqref{pf-CLT-equivalence-tau-u}, \eqref{pf-CLT-equivalence-tau-u-y},
Propositions \ref{Prop_CLT_aaa} and \ref{Prop_CLT_bbb}. 
\end{proof}

\section{Proof of Theorems \ref{Thm_LD_Perpe_Petrov} and \ref{Cor_LLT_LD_tau}}\label{Sec-Pf-Thms-5-8}

The goal of this section is to establish Theorem \ref{Thm_LD_Perpe_Petrov}
on precise large deviations for the first passage times $\tau_u$ and $\tau_u^y$, 
and Theorem \ref{Cor_LLT_LD_tau} on local limit theorems for $\tau_u$ and $\tau_u^y$. 
We shall need Lemma \ref{Lem_Martin}. 

\begin{proof}[Proof of Lemma \ref{Lem_Martin}]
We start by proving that 
\begin{align*}
W_n(s): =  \frac{1}{ \kappa(s)^n }  |V_n^*|^s r_s \left( \frac{V_n^*}{|V_n^*|} \right)
\end{align*}
is a submartingale with respect to the filtration $\mathscr{F}_n : = \sigma \left( (M_i, Q_i)_{i=1}^n \right)$. 
In view of \eqref{Def_r_s}, we have
$r_{s}(x)= \int_{\bb S_+^{d-1}} \langle x, y\rangle^s  \nu^*_s(dy)$ and hence
\begin{align*}
\bb E ( W_{n+1}(s) | \mathscr{F}_n )
& = \frac{1}{ \kappa(s)^{n+1} } 
 \bb E \left( \left.  |V_{n+1}^*|^s 
    \int_{\bb S_+^{d-1}} \left\langle \frac{ V_{n+1}^* }{ |V_{n+1}^*| }, y \right\rangle^s  \nu^*_s(dy)   \right| \mathscr{F}_n \right)
 \nonumber\\
& =  \frac{1}{ \kappa(s)^{n+1} } 
 \bb E \left(  \left.  \int_{\bb S_+^{d-1}} \langle  V_{n+1}^*, y\rangle^s  \nu^*_s(dy) \right| \mathscr{F}_n \right). 
\end{align*}
By \eqref{Recur_Equ}, we have $V_{n+1}^* = M_{n+1} V_n^* + Q_{n+1}$. 
Since each component $Q_n$ is non-negative and $P_s r_s = \kappa(s) r_s$, we get
\begin{align*}
\bb E ( W_{n+1}(s) | \mathscr{F}_n ) 
& \geq  \frac{1}{ \kappa(s)^{n+1} } 
 \bb E \left( \left.  \int_{\bb S_+^{d-1}}  \langle M_{n+1} V_n^*, y\rangle^s  \nu^*_s(dy) \right| \mathscr{F}_n \right)
 \nonumber\\
&  =  \frac{1}{ \kappa(s)^{n+1} } 
 \bb E \left[ \left. r_s \left( \frac{M_{n+1} V_n^*}{|M_{n+1} V_n^*|} \right) |M_{n+1} V_n^*|^s  \right| \mathscr{F}_n \right]
 \nonumber\\
&  =  \frac{1}{ \kappa(s)^{n+1} } 
 \bb E \left[ \left.  r_s \left( \frac{M_{n+1} V_n^*}{|M_{n+1} V_n^*|} \right) 
    \left| M_{n+1} \frac{V_n^*}{|V_n^*|} \right|^s  |V_n^*|^s  \right| V_n^* \right]   \nonumber\\
&  =  \frac{1}{ \kappa(s)^{n+1} }  |V_n^*|^s  P_s r_s \left( \frac{V_n^*}{|V_n^*|} \right)   \nonumber\\
&  =  \frac{1}{ \kappa(s)^n }  |V_n^*|^s  r_s \left( \frac{V_n^*}{|V_n^*|} \right)  = W_n(s).  
\end{align*}
This shows that $(W_n(s), \mathscr{F}_n)$ is a submartingale.

We next prove $\sup_{n \geq 1} \bb E W_n(s) < \infty$. 
Using the fact that $r_s$ is bounded on $\bb S_+^{d-1}$, 
in view of \eqref{Def_Yn02}, we obtain that if $0< s \leq 1$, then
\begin{align*}
\bb E W_n(s) \leq \frac{c}{ \kappa(s)^n } \bb E \left( |V_n^*|^s \right)
      \leq \frac{c}{ \kappa(s)^n } \sum_{j=1}^n \bb E \left( \|M_n \ldots M_{j+1} \|^s |Q_j|^s \right).   
\end{align*}
By independence and Lemma \ref{Lem_kappa}, it follows that 
\begin{align*}
\bb E W_n(s) \leq  \frac{c_s}{ \kappa(s)^n } \bb E (|Q_1|^s) \sum_{j=1}^n  \kappa(s)^j
      =  c_s \bb E (|Q_1|^s) \frac{\kappa(s)}{\kappa(s) - 1}  \frac{ \kappa(s)^n - 1 }{ \kappa(s)^n }  \leq  c_s,
\end{align*}
where in the last inequality we used $\bb E (|Q_1|^s) < \infty$ and the fact that $\kappa(s) > 1$ for any $s > \alpha$. 
This shows that $\sup_{n \geq 1} \bb E W_n(s) < \infty$ if $0< s \leq 1$.
If $s>1$, we apply Minkowski's inequality, Lemma \ref{Lem_kappa} and $\bb E (|Q_1|^s) < \infty$ to get that 
\begin{align*}
 \big[\bb E (|V_n^*|^s) \big]^{ \frac{1}{s} }  
   & \leq  \sum_{j=1}^n  \Big[ \bb E \big( |(M_n \ldots M_{j+1}) Q_j|^s  \big) \Big]^{ \frac{1}{s} }  \nonumber\\
   &  \leq  \sum_{j=1}^n  \Big[ \bb E ( \|M_n \ldots M_{j+1} \|^s ) \bb E(|Q_j|^s)  \Big]^{ \frac{1}{s} }  \nonumber\\
   &  \leq  c_s \big[ \bb E(|Q_1|^s) \big]^{ \frac{1}{s} } \sum_{j=1}^n    \kappa(s)^{\frac{n-j}{s}}  \nonumber\\
   &  \leq  c_s  \frac{ \kappa(s)^{ \frac{n}{s} } - 1 }{\kappa(s)^{ \frac{1}{s} } - 1}  \notag\\
   & \leq  c_s  \kappa(s)^{ \frac{n}{s} },
\end{align*}
where in the last inequality we used the fact that $\kappa(s) > 1$ for any $s > \alpha$. 
This, together with the boundedness of the eigenfunction $r_s$, implies
\begin{align*}
( \bb E W_n(s) )^{ \frac{1}{s} }  \leq \frac{c}{ \kappa(s)^{ \frac{n}{s} } } \big[\bb E (|V_n^*|^s) \big]^{ \frac{1}{s} }
  \leq c_s, 
\end{align*}
as desired. 

We finally show $\lim_{ n \to \infty } W_n(s) = W(s)$ a.s. with $\bb E W(s) < \infty$, and \eqref{Limit-Vn-Vnstar}. 
Since $(W_n(s), \mathscr{F}_n)$ is a submartingale and $\sup_{n \geq 1} \bb E W_n(s) < \infty$,
the convergence $W_n(s) \to W(s)$ a.s. follows from Doob's martingale convergence theorem. 
Since $W_n(s) \geq 0$, 
using Fatou's lemma, we get
\begin{align*}
\bb E W(s) = \bb E \left( \liminf_{n \to \infty} W_n(s) \right) \leq \liminf_{n \to \infty} \bb E W_n(s)
\leq \sup_{n \geq 1} \bb E W_n(s) < \infty. 
\end{align*}
The equality in \eqref{Limit-Vn-Vnstar} holds since $V_n$ and $V_n^*$ have the same distribution. 
The limit in \eqref{Limit-Vn-Vnstar} exists and is finite due to the submartingale property of $(W_n(s), \mathscr{F}_n)$ 
and the fact that $\sup_{n \geq 1} \bb E W_n(s) < \infty$. 
By condition \ref{Condi_ExpMom}, we have $\bb P (Q_1 = 0) < 1$, 
so that $\bb E W_1(s) > 0$ by using the fact that $r_s$ is strictly positive on $\bb S_+^{d-1}$. 
This, combined with $\bb E W_n(s) \leq \bb E W_{n+1}(s)$, shows that the limit in \eqref{Limit-Vn-Vnstar} is strictly positive. 
\end{proof}

\subsection{Proof of Theorem \ref{Thm_LD_Perpe_Petrov}}
In view of Lemma \ref{Lem_tau_Low}, it suffices to prove that for any $a > \frac{3 + 2s}{2\Lambda(s)}$ and $\epsilon >0$,
we have, as $u \to \infty$, uniformly in $l \in [0, (\log u)^{-\epsilon}]$, 
\begin{align}\label{LD-pf-tau-u}
\bb P \Big( (\beta - l) \log u - a \log \log u < \tau_u \leq (\beta - l) \log u \Big) 
= \frac{ \mathscr C_{\beta, l} (u) }{\sqrt{\log u}} u^{-I(\beta - l)} (1 + o(1)),
\end{align}
and uniformly in $y \in \bb S_+^{d-1}$, 
\begin{align}\label{LD-pf-tau-u-y}
\bb P \Big( (\beta - l) \log u - a \log \log u < \tau_u^y \leq (\beta - l) \log u \Big) 
= r_s^*(y) \frac{ \mathscr C_{\beta, l} (u) }{\sqrt{\log u}} u^{-I(\beta - l)} (1 + o(1)). 
\end{align}
We first prove \eqref{LD-pf-tau-u}. 
For brevity, we denote
\begin{align}\label{def-upzeta-u-beta-u}
\upzeta_u = \floor{ (\beta - l) \log u - a \log \log u},  \qquad  \beta_u = \floor{ (\beta - l) \log u}
\end{align}
and 
\begin{align}\label{Pf_Def_Yu}
\mathcal{Y}_u  = Q_{\upzeta_u + 1} + M_{\upzeta_u + 1} Q_{\upzeta_u + 2} + \ldots 
    + (M_{\upzeta_u + 1} \ldots M_{\beta_u - 1}) Q_{\beta_u}.  
\end{align}
Then the following decomposition holds: 
\begin{align}\label{Pf_Yn_Equ01}
V_{\beta_u} = V_{\upzeta_u} + \Pi_{\upzeta_u} \mathcal{Y}_u. 
\end{align}
Let us first give a precise large deviation asymptotic for the second part $\Pi_{\upzeta_u} \mathcal{Y}_u$, 
and we will see later that the contribution of the first part $V_{\upzeta_u}$ can be neglected. 
Denote
\begin{align*}
\mathcal{X}_u  = \frac{ \mathcal{Y}_u }{ |\mathcal{Y}_u| }, 
\end{align*}
then $\mathcal{X}_u \in \bb S_+^{d-1}$ and it holds that
\begin{align}
& \bb P (|\Pi_{\upzeta_u} \mathcal{Y}_u| > u) 
   = \bb P( \log |\Pi_{\upzeta_u} \mathcal{X}_u| > \log u - \log |\mathcal{Y}_u|),   \label{Pf_Identi001}  \\
&  \bb P ( \langle y, \Pi_{\upzeta_u} \mathcal{Y}_u \rangle > u) 
   = \bb P( \log \langle y, \Pi_{\upzeta_u} \mathcal{X}_u \rangle > \log u - \log |\mathcal{Y}_u|).   \label{Pf_Identi001-y}
\end{align}
For any starting point $x \in \bb S_+^{d-1}$, 
the inequality $\log |\Pi_{\upzeta_u} x| > \log u - w$ holds if and only if 
\begin{align}\label{Pf_LD_luy}
\frac{ \log |\Pi_{\upzeta_u} x| }{ \upzeta_u } > \frac{ \log u - w }{ \upzeta_u }
 : = \frac{1}{\beta - l} + l_{u,w},
\end{align}
where, by \eqref{def-upzeta-u-beta-u}, 
\begin{align*}
l_{u,w}  = \frac{ \log u }{\upzeta_u} - \frac{1}{\beta - l}  -  \frac{w}{ \upzeta_u }
          =  \delta_u  \frac{\log \log u}{\log u} - \frac{w}{ \upzeta_u }, 
\end{align*}
with $0< \delta_u < \frac{2a}{\beta}$ for large enough $u > 0$. 
Then it holds that $l_{u,w} \to 0$ and $\upzeta_u l_{u,w}^2 \to 0$ as $u \to \infty$,  
uniformly in $w \in [- (\log u)^{\epsilon/2}, (\log u)^{\epsilon/2}]$
and $l \in [0, (\log u)^{-\epsilon}]$. 
Now we apply \eqref{LDPet_VectorNorm} of Theorem \ref{PetrovThm} with $n = \upzeta_u$, $q = \Lambda'(s) = \frac{1}{\beta}$
and $l_1 = \frac{1}{\beta - l} - \frac{1}{\beta} + l_{u,w}$ to 
obtain that as $u \to \infty$, 
uniformly in $x \in \bb S_+^{d-1}$, $|w| \leq (\log u)^{\epsilon/2}$ and $l \in [0, (\log u)^{-\epsilon}]$, 
\begin{align}\label{Pf_LD_PRM}
& \mathbb{P} \big( \log |\Pi_{\upzeta_u} x| > \log u - w \big)   \nonumber\\
&  =  \frac{ r_s(x) }{s \sigma_s  \nu_s(r_s) \sqrt{2 \pi \upzeta_u}}
    \exp \left\{ - \upzeta_u \Lambda^* \left( \frac{1}{\beta - l} + l_{u,w} \right)  \right\} ( 1 + o(1)).
\end{align}
As in \eqref{beta-l-Lambda}, 
there exists $l_s \geq 0$ satisfying $l_s = O((\log u)^{-\epsilon})$ such that $(\beta - l)^{-1} = \Lambda'(s + l_s).$
Hence, using \eqref{Def Jsl}, we get
\begin{align*}
& \Lambda^* \left( \frac{1}{\beta - l} + l_{u,w} \right)  
 = \Lambda^* \left( \frac{1}{\beta - l} \right) + (s+l_s) l_{u,w} + \frac{l_{u,w}^2}{ 2 \sigma_{s+l_s}^2 }
   - \frac{l_{u,w}^3}{ \sigma_{s+l_s}^3 } \xi_{s+l_s} \left(\frac{ l_{u,w} }{\sigma_{s+l_s}} \right). 
\end{align*}
Since $\upzeta_u l_{u,w}^2 \to 0$ as $u \to \infty$, this yields
that uniformly in $|y| \leq (\log u)^{\epsilon/2}$ and $l \in [0, (\log u)^{-\epsilon}]$, 
\begin{align}\label{Expansion-upzeta-aa}
\upzeta_u \Lambda^* \left( \frac{1}{\beta - l} + l_{u,w} \right)
= \upzeta_u \left[  \Lambda^* \left( \frac{1}{\beta - l} \right) + (s+l_s) l_{u,w} \right] (1+o(1)).
\end{align}
By the definition of $\Lambda^*$ and \eqref{Pf_LD_luy}, it follows that 
\begin{align}\label{Expansion-upzeta-bb}
 \upzeta_u \left[  \Lambda^* \left( \frac{1}{\beta - l} \right) + (s+l_s) l_{u,w} \right]   
& = \upzeta_u \left[ (s+l_s) \frac{1}{\beta - l} - \Lambda(s+l_s) + (s+l_s) l_{u,w} \right]  \nonumber\\
& = (s+l_s)(\log u - y) - \upzeta_u \Lambda(s+l_s). 
\end{align}
Using $\upzeta_u = \beta \log u (1+o(1))$ as $u \to \infty$, 
and substituting \eqref{Expansion-upzeta-aa} and \eqref{Expansion-upzeta-bb} into \eqref{Pf_LD_PRM},
we get that
as $u \to \infty$, 
uniformly in $x \in \bb S_+^{d-1}$, $|w| \leq (\log u)^{\epsilon/2}$ and $l \in [0, (\log u)^{-\epsilon}]$, 
\begin{align}
 \mathbb{P} \big( \log |\Pi_{\upzeta_u} x| > \log u - w \big)   
&  = \frac{ r_s(x) }{s \sigma_s  \nu_s(r_s) \sqrt{2 \pi \beta \log u}}
    u^{-(s+l_s)} e^{(s+l_s)w}  \kappa(s+l_s)^{\upzeta_u} ( 1 + o(1))  \nonumber\\
&  =  \frac{ r_s(x) u^{- I(\beta - l)} }{s \sigma_s  \nu_s(r_s) \sqrt{2 \pi \beta \log u}}
     e^{(s+l_s)w}  \kappa(s+l_s)^{\upzeta_u - (\beta-l) \log u} ( 1 + o(1))  \nonumber\\
&  =  \frac{ r_s(x) u^{- I(\beta - l)} }{s \sigma_s  \nu_s(r_s) \sqrt{2 \pi \beta \log u}}
     e^{sw} \kappa(s+l_s)^{\upzeta_u - (\beta-l) \log u} ( 1 + o(1)),  \label{Pf_LD_PRM02}
\end{align}
where in the second equality we used \eqref{Pf_Iden_Ibeta}
and in the last equality we used the fact that $l_s = O((\log u)^{-\epsilon})$ so that
$l_s w \to 0$ as $u \to \infty$, uniformly in $|w| \leq (\log u)^{\epsilon/2}$.  
Similarly, following the proof of \eqref{Pf_LD_PRM02}, one can use \eqref{LDPet_ScalarProduct} instead of \eqref{LDPet_VectorNorm}
to obtain that as $u \to \infty$, 
uniformly in $x, y \in \bb S_+^{d-1}$, $|w| \leq (\log u)^{\epsilon/2}$ and $l \in [0, (\log u)^{-\epsilon}]$, 
\begin{align}
 \mathbb{P} \big( \log \langle y, \Pi_{\upzeta_u} x \rangle > \log u - w \big)   
  =  \frac{ r_s(x) r_s^*(y) u^{- I(\beta - l)} }{s \sigma_s  \nu_s(r_s) \sqrt{2 \pi \beta \log u}}
     e^{sw} \kappa(s+l_s)^{\upzeta_u - (\beta-l) \log u} ( 1 + o(1)).  \label{Pf_LD_PRM02-y}
\end{align}
For brevity, we denote
\begin{align*}
B_u = \left\{ \log |\mathcal{Y}_u| \in [- (\log u)^{\epsilon/2}, (\log u)^{\epsilon/2}] \right\} 
\end{align*}
and write $B_u^c$ for its complement.
Then we have the following decompositions: 
\begin{align}\label{Pf_Thm1_Decom_P}
\bb P (|\Pi_{\upzeta_u} \mathcal{Y}_u| > u) 
 = \bb P (|\Pi_{\upzeta_u} \mathcal{Y}_u| > u, \,  B_u)  + \bb P (|\Pi_{\upzeta_u} \mathcal{Y}_u| > u, \,  B_u^c), 
\end{align}
and for any $y \in \bb S_+^{d-1}$, 
\begin{align}\label{Pf_Thm1_Decom_P-y}
\bb P ( \langle y, \Pi_{\upzeta_u} \mathcal{Y}_u \rangle > u) 
 = \bb P ( \langle y, \Pi_{\upzeta_u} \mathcal{Y}_u \rangle > u, \,  B_u)  
    + \bb P ( \langle y, \Pi_{\upzeta_u} \mathcal{Y}_u \rangle > u, \,  B_u^c). 
\end{align}
Now we deal with the first terms in \eqref{Pf_Thm1_Decom_P} and \eqref{Pf_Thm1_Decom_P-y}. 
Recall that $\chi_{\beta, l} (u)$ is defined in \eqref{def-chi-C-beta-l-u}. 

\begin{lemma}\label{Lem-LD-main}
Let $\beta \in (0, \rho)$ with $\rho$ defined by \eqref{Def_rho}. 
If \ref{Condi_ExpMom}, \ref{Condi_AP} and \ref{Condi_NonArith} hold, 
then for any $\epsilon >0$ and $a > \frac{3 + 2s}{2\Lambda(s)}$, as $u \to \infty$, uniformly in $l \in [0, (\log u)^{-\epsilon}]$, 
\begin{align}\label{Pf_Thm1_Yu_Au} 
\bb P (|\Pi_{\upzeta_u} \mathcal{Y}_u| > u, \,  B_u)   
  = \mathscr H_s    \frac{ \kappa(s)^{-\chi_{\beta, l} (u) }  }{ \sqrt{\log u} }
    u^{-I(\beta - l)} ( 1 + o(1)),  
\end{align}
where, with $D_n : = \{ \log |V_n| \in [- e^{\epsilon n/2a}, e^{\epsilon n/2a}] \}$, 
\begin{align}\label{def-Hs-Bn}
\mathscr H_s 
=  \frac{1}{s \sigma_s \nu_s(r_s) \sqrt{2 \pi \beta}}
     \lim_{n \to \infty} \frac{1}{ \kappa(s)^n } 
    \bb E \Big[ r_s \Big( \frac{V_n}{|V_n|} \Big) |V_n|^s; D_n \Big]. 
\end{align}
If \ref{Condi_ExpMom}, \ref{Condi_NonArith} and \ref{Condi_Kes_Weak} hold, 
then for any $\epsilon >0$ and $a > \frac{3 + 2s}{2\Lambda(s)}$, 
as $u \to \infty$, uniformly in $y \in \bb S_+^{d-1}$ and $l \in [0, (\log u)^{-\epsilon}]$, 
\begin{align}\label{Pf_Thm1_Yu_Au-y} 
\bb P ( \langle y, \Pi_{\upzeta_u} \mathcal{Y}_u \rangle > u, \,  B_u)   
  =  r_s^*(y) \mathscr H_s    \frac{ \kappa(s)^{-\chi_{\beta, l} (u) }  }{ \sqrt{\log u} }
    u^{-I(\beta - l)} ( 1 + o(1)). 
\end{align}
\end{lemma}

\begin{proof}
We first note that $\mathscr H_s$ is finite, by Lemma \ref{Lem_Martin} and the positivity of $r_s$. 
Using \eqref{Pf_Identi001} and \eqref{Pf_LD_PRM02} with $x = \mathcal{X}_u$ and $w = \log |\mathcal{Y}_u|$, 
we get that, as $u \to \infty$, uniformly in $l \in [0, (\log u)^{-\epsilon}]$, 
\begin{align}
 \bb P (|\Pi_{\upzeta_u} \mathcal{Y}_u| > u, \,  B_u)   
& =  \bb P( \log |\Pi_{\upzeta_u} \mathcal{X}_u| > \log u - \log |\mathcal{Y}_u|, \,  B_u)   \notag\\
& = \int_{B_u} \frac{ r_s( \mathcal{X}_u ) u^{- I(\beta - l)} }{s \sigma_s \nu_s(r_s) \sqrt{2 \pi \beta \log u}}
     |\mathcal{Y}_u|^s  \kappa(s+l_s)^{\upzeta_u - (\beta-l) \log u}
     d \bb P ( 1 + o(1))  \nonumber\\
& = \bb E \Big[ r_s( \mathcal{X}_u ) |\mathcal{Y}_u|^s; B_u \Big]
    \frac{ u^{- I(\beta - l)}  \kappa(s+l_s)^{\upzeta_u - (\beta-l) \log u} }
    {s \sigma_s \nu_s(r_s) \sqrt{2 \pi \beta \log u}}
     ( 1 + o(1))          \nonumber\\
&  = \frac{1}{ \kappa(s+l_s)^{\beta_u - \upzeta_u} } 
 \bb E \Big[ r_s( \mathcal{X}_u ) |\mathcal{Y}_u|^s; B_u \Big]   \nonumber\\
&  \qquad \times    \frac{ u^{-I(\beta - l)}  \kappa(s+l_s)^{-(\beta - l) \log u + \floor{(\beta - l) \log u} } }
    {s \sigma_s \nu_s(r_s) \sqrt{2 \pi \beta \log u}}
     ( 1 + o(1)),  \label{Pf_Thm1_Yu_Au_a} 
\end{align}
where in the last line we used $\beta_u = \floor{ (\beta - l) \log u}$. 
Since $|l_s| \leq c (\log u)^{-\epsilon}$, by Taylor's formula, one can verify that
\begin{align}\label{Approxi-kappa-upzeta}
\lim_{u \to \infty} \frac{ \kappa(s)^{\beta_u - \upzeta_u} }{ \kappa(s+l_s)^{\beta_u - \upzeta_u}} = 1. 
\end{align}
In view of \eqref{Def_Yn01} and \eqref{Pf_Def_Yu}, we see that
the distribution of $\mathcal{Y}_u$ coincides with that of the random variable $V_{\floor{ \beta_u - \upzeta_u } }$, so that
\begin{align}\label{limit-Hs-Bn}
& \frac{1}{s \sigma_s \nu_s(r_s) \sqrt{2 \pi \beta}}
\lim_{u \to \infty} \frac{1}{ \kappa(s)^{ \beta_u - \upzeta_u }} 
  \bb E \Big[ r_s( \mathcal{X}_u ) |\mathcal{Y}_u|^s; B_u \Big]   \notag\\
&  = \frac{1}{s \sigma_s \nu_s(r_s) \sqrt{2 \pi \beta}} \lim_{u \to \infty} \frac{1}{ \kappa(s)^{ \beta_u - \upzeta_u } } 
    \bb E \Big[ r_s \Big( \frac{V_{ \beta_u - \upzeta_u } }{|V_{ \beta_u - \upzeta_u } |} \Big) |V_{ \beta_u - \upzeta_u } |^s; B_u' \Big] 
    = \mathscr H_s,  
\end{align}
where $B_u' = \{ \log | V_{\floor{ \beta_u - \upzeta_u } } | \in [- (\log u)^{\epsilon/2}, (\log u)^{\epsilon/2}] \}$. 
Therefore, combining \eqref{Pf_Thm1_Yu_Au_a}, \eqref{Approxi-kappa-upzeta} and \eqref{limit-Hs-Bn}, 
we conclude the proof of \eqref{Pf_Thm1_Yu_Au}. 
The proof of \eqref{Pf_Thm1_Yu_Au-y} is similar by using \eqref{Pf_LD_PRM02-y} instead of \eqref{Pf_LD_PRM02}. 
\end{proof}

Next we show that the second terms in \eqref{Pf_Thm1_Decom_P} and \eqref{Pf_Thm1_Decom_P-y} are negligible compared with the first terms. 
In the proof, the condition $s \in \frac{1}{2} I_{\mu}^{\circ}$ is used. 

\begin{lemma}\label{Lem-LD-smallterm}
Assume \ref{Condi_ExpMom}, \ref{Condi_AP} and \ref{Condi_NonArith}. 
Let $\beta = \frac{1}{\Lambda'(s)}$ be such that $2 s \in I_{\mu}^{\circ}$ and $\beta \in (0, \rho)$, with $\rho$ defined by \eqref{Def_rho}.
Then, as $u \to \infty$, uniformly in $y \in \bb S_+^{d-1}$ and $l \in [0, (\log u)^{-\epsilon}]$,  
\begin{align}\label{Pf_Thm1_Yu_Bu} 
\bb P (|\Pi_{\upzeta_u} \mathcal{Y}_u| > u, \,  B_u^c)   
  = o \left( \frac{ u^{-I(\beta - l)} }{\sqrt{\log u}}  \right),  
  \qquad 
  \bb P ( \langle y, \Pi_{\upzeta_u} \mathcal{Y}_u \rangle > u, \,  B_u^c)   
  = o \left( \frac{ u^{-I(\beta - l)} }{\sqrt{\log u}}  \right).  
\end{align}
\end{lemma}

\begin{proof}
It suffices to prove the first asymptotic in \eqref{Pf_Thm1_Yu_Bu} since $\langle y, \Pi_{\upzeta_u} \mathcal{Y}_u \rangle \leq |\Pi_{\upzeta_u} \mathcal{Y}_u|$. 
By the definition of $B_u^c$, 
\begin{align}\label{Second-Bu-decom}
\bb P (|\Pi_{\upzeta_u} \mathcal{Y}_u| > u, \,  B_u^c) 
& =  \bb P \left(|\Pi_{\upzeta_u} \mathcal{Y}_u| > u, \,  \log |\mathcal{Y}_u| < - (\log u)^{\epsilon/2} \right)  
  +  \bb P \big( |\Pi_{\upzeta_u} \mathcal{Y}_u| > u, \,  \log |\mathcal{Y}_u| >  (\log u)^{\epsilon/2} \big)  \notag\\
 & = : I_1(u) + I_2(u). 
\end{align}
For $I_1(u)$, we have 
\begin{align*}
I_1(u) 
& \leq  \bb P \left( \log \|\Pi_{\upzeta_u} \|  > \log u - \log |\mathcal{Y}_u|,  \,  \log |\mathcal{Y}_u| < - (\log u)^{\epsilon/2} \right) 
   \nonumber\\
& \leq  \bb P \left( \log \|\Pi_{\upzeta_u} \| > \log u +  (\log u)^{\epsilon/2} \right)   \nonumber\\
&  =  \bb P \left( \|\Pi_{\upzeta_u} \| > \exp \left\{ \log u +  (\log u)^{\epsilon/2} \right\} \right). 
\end{align*}
By Markov's inequality, Lemma \ref{Lem_kappa} and the fact that $\upzeta_u \leq (\beta - l) \log u$, it follows that
\begin{align}\label{Second-Bu-bound-I1}
I_1(u) 
&  \leq  \exp \left\{ - (s+l_s) \log u -  (s + l_s)(\log u)^{\epsilon/2} \right\}  
  \bb E \big( \|\Pi_{\upzeta_u} \|^{s+l_s} \big)   \nonumber\\
&  \leq  c  \exp \left\{ -(s+l_s) \log u -  (s+l_s)(\log u)^{\epsilon/2}  \right\}  \kappa(s+l_s)^{(\beta - l) \log u}   \nonumber\\
&  =  o \left( \frac{1}{\sqrt{\log u}}  \kappa(s+l_s)^{(\beta - l) \log u} u^{-(s+l_s)} \right)   \nonumber\\
&  =  o \left( \frac{ u^{-I(\beta - l)} }{\sqrt{\log u}}   \right),  
\end{align}
where in the last line we used \eqref{Pf_Iden_Ibeta}.

For $I_2(u)$, 
since $\log |\Pi_{\upzeta_u} \mathcal{Y}_u| \leq \log \|\Pi_{\upzeta_u} \| + \log |\mathcal{Y}_u|$ 
and $\Pi_{\upzeta_u}$ is independent of $\mathcal{Y}_u$,  it holds that
\begin{align*}
I_2(u) 
&  = \sum_{m=1}^{\infty}  \bb P \left( |\Pi_{\upzeta_u} \mathcal{Y}_u| > u,  \,  
     (m+1) (\log u)^{\epsilon/2} \geq  \log |\mathcal{Y}_u| >  m (\log u)^{\epsilon/2}  \right)  \nonumber\\
&  \leq \sum_{m=1}^{\infty}  \bb P \Big( \log \|\Pi_{\upzeta_u} \| > \log u - (m+1) (\log u)^{\epsilon/2}, \,  
             \log |\mathcal{Y}_u| >  m (\log u)^{\epsilon/2}  \Big)  \nonumber\\
&  =  \sum_{m=1}^{\infty} \bb P \left( \| \Pi_{\upzeta_u} \| >  u \exp \left\{ - (m+1) (\log u)^{\epsilon/2} \right\}  \right)
                \bb P \left(  \log |\mathcal{Y}_u| >  m (\log u)^{\epsilon/2}  \right).
\end{align*}
By Markov's inequality and Lemma \ref{Lem_kappa}, we get
\begin{align*}
&  \bb P \left( \| \Pi_{\upzeta_u} \| >  u \exp \left\{ - (m+1) (\log u)^{\epsilon/2} \right\} \right)   \nonumber\\
& \leq   \exp \left\{ (s + l_s) (m+1) (\log u)^{\epsilon/2} \right\}  u^{-(s + l_s)} \bb E \left( \|\Pi_{\upzeta_u}\|^{s + l_s} \right)  \nonumber\\
& \leq  c \, \exp \left\{ (s + l_s) (m+1) (\log u)^{\epsilon/2} \right\}  u^{-(s + l_s)}  \kappa(s + l_s)^{\upzeta_u}. 
\end{align*}
Since $\upzeta_u \leq (\beta - l) \log u$, 
by \eqref{Pf_Iden_Ibeta}  it follows that 
\begin{align*}
 \bb P \left( \| \Pi_{\upzeta_u} \| >  u \exp \left\{ - (m+1) (\log u)^{\epsilon/2} \right\} \right)  
 \leq  c \,  u^{-I(\beta - l)}  
     \exp \left\{ s (m+1) (\log u)^{\epsilon/2} \right\}. 
\end{align*}
Since $2s \in I_{\mu}^{\circ}$,  we can choose $t > s$
and apply \eqref{reamrk-Vn-u01} of Remark \ref{Rem-Inequality-Vn} with 
$n = \floor{ \beta_u - \upzeta_u } \leq a \log\log u$ and $u = \exp \big\{ m (\log u)^{\epsilon/2} \big\}$ to get
that there exists $s_1 \in (s, s+t)$ such that 
\begin{align*}
 \bb P \Big(  \log |\mathcal{Y}_u| >  m (\log u)^{\epsilon/2}  \Big)  
& \leq  \bb P \Big( |V_{ \floor{ \beta_u - \upzeta_u } }| >  \exp \big\{ m (\log u)^{\epsilon/2} \big\}  \Big)  \nonumber\\
& \leq  c ( \log\log u)^{1+ s +t}  
     \exp \left\{ a \left( t \Lambda'(s) + \frac{1}{2} t^2 \Lambda''(s_1) \right) \log\log u \right\}    \nonumber\\
& \quad \times  \kappa(s)^{a \log\log u}  \exp \left\{ - (s+t) m (\log u)^{\epsilon/2} \right\}. 
\end{align*}
Combining the above estimates and summing up over $m$, we derive that there  exist constants $M_1, M_2 >0$ such that 
\begin{align}\label{Second-Bu-bound-I2}
I_2(u) 
& \leq  c \,  u^{-I(\beta - l)}   (\log\log u)^{1+s+t}  \kappa(s)^{a \log\log u}    
    \exp \left\{ a \left( t \Lambda'(s) + \frac{1}{2} t^2 \Lambda''(s_1) \right) \log\log u \right\}   \nonumber\\
&   \qquad \times   \sum_{m=1}^{\infty}  \exp \left\{ (s - tm) (\log u)^{\epsilon/2} \right\}      \nonumber\\
&  \leq  c \, u^{-I(\beta - l)}  (\log u)^{M_1}  \exp \left\{ (s-t) (\log u)^{\epsilon/2} \right\}      \nonumber\\
& \leq  c \, u^{-I(\beta - l)} (\log u)^{-M_2},
\end{align}
where in the last inequality we used $t >s$. 
Substituting \eqref{Second-Bu-bound-I1} and \eqref{Second-Bu-bound-I2} into \eqref{Second-Bu-decom} 
concludes the proof of \eqref{Pf_Thm1_Yu_Bu}. 
\end{proof}

\begin{proof}[End of the proof of Theorem \ref{Thm_LD_Perpe_Petrov}]
We first prove \eqref{Petrov_LD_tau}. 
Combining \eqref{Pf_Thm1_Decom_P}, \eqref{Pf_Thm1_Yu_Au} and \eqref{Pf_Thm1_Yu_Bu}, 
we get that, as $u \to \infty$, uniformly in $l \in [0, (\log u)^{-\epsilon}]$, 
\begin{align*}
\bb P (|\Pi_{\upzeta_u} \mathcal{Y}_u| > u )   
  = \mathscr H_s    \frac{ \kappa(s)^{-\chi_{\beta, l} (u) }  }{ \sqrt{\log u} }
    u^{-I(\beta - l)} ( 1 + o(1)),  
\end{align*}
where $\mathscr H_s$ and $\chi_{\beta, l} (u)$ are defined by \eqref{def-Hs-Bn} and \eqref{def-chi-C-beta-l-u}, respectively. 
Now we prove that $\mathscr H_s = \varkappa_s$ with $\varkappa_s$ defined by \eqref{Def_C_beta}. 
It suffices to show that 
\begin{align}
J_1 &:= \lim_{n \to \infty} \frac{1}{ \kappa(s)^n } 
    \bb E \Big[ r_s \Big( \frac{V_n}{|V_n|} \Big) |V_n|^s;  \,  \log |V_n| < - e^{\epsilon n/2a}  \Big]  = 0,   \label{limit-kappa-s-aa}\\
J_2 &:=  \lim_{n \to \infty} \frac{1}{ \kappa(s)^n } 
    \bb E \Big[ r_s \Big( \frac{V_n}{|V_n|} \Big) |V_n|^s;  \,  \log |V_n| >  e^{\epsilon n/2a}  \Big]  = 0.   \label{limit-kappa-s-bb}
\end{align}
For \eqref{limit-kappa-s-aa},  since the eigenfunction $r_s$ is bounded on $\bb S_+^{d-1}$, we see that
\begin{align*}
J_1 \leq  \lim_{n \to \infty} \frac{c}{\kappa(s)^n} \bb E \Big[ |V_n|^s; \log |V_n| < - e^{\epsilon n/2a} \Big]  
  \leq  c \lim_{n \to \infty}  e^{-n \Lambda(s)} \exp \left\{ - s e^{\epsilon n/2a} \right\} = 0. 
\end{align*}
For \eqref{limit-kappa-s-bb}, using again that $r_s$ is bounded on $\bb S_+^{d-1}$, we get
\begin{align}\label{pf-J2-upper-bound}
J_2  & \leq \lim_{n \to \infty} \frac{c}{\kappa(s)^n} \bb E \Big[ |V_n|^s; \log |V_n| > e^{\epsilon n/2a} \Big]   \notag\\
& = \lim_{n \to \infty} \frac{c}{\kappa(s)^n}  \sum_{m=0}^{\infty}
     \bb E \Big[ |V_n|^s; \log |V_n| - e^{\epsilon n/2a} \in (m, m+1] \Big]   \notag\\
& \leq  \lim_{n \to \infty} \frac{c}{\kappa(s)^n}  \exp \left\{ s e^{\epsilon n/2a} \right\} 
  \sum_{m=0}^{\infty}  e^{s (m+1)} 
    \bb P \left( |V_n| > e^m \exp \{ e^{\epsilon n/2a} \} \right). 
\end{align}
By Lemma \ref{Lem_Yn_u}, there exists a constant $\ee_1 >0$ such that 
\begin{align*}
 \bb P \left( |V_n| > e^m \exp \{ e^{\epsilon n/2a} \} \right)   
 \leq c \,  n^{1+ s + \ee_1} 
\exp \left\{ n (\ee_1 \Lambda'(s) + \ee_1^2 \Lambda''(s) ) \right\}  \kappa(s)^n  
    e^{ -(s+\ee_1) m} \exp \left\{ -(s+\ee_1) e^{\epsilon n/2a} \right\}.
\end{align*}
Substituting this into \eqref{pf-J2-upper-bound} yields that 
\begin{align*}
J_2   \leq  c \sum_{m=0}^{\infty}  e^{ - \ee_1 m}  \lim_{n \to \infty}  n^{1+ s + \ee_1} 
\exp \left\{ n (\ee_1 \Lambda'(s) + \ee_1^2 \Lambda''(s) ) \right\}  
  \exp \left\{ - \ee_1 e^{\epsilon n/2a} \right\}
  = 0, 
\end{align*}
which proves \eqref{limit-kappa-s-bb}. 
Therefore,   $\mathscr H_s = \varkappa_s$ and, as $u \to \infty$, uniformly in $l \in [0, (\log u)^{-\epsilon}]$, 
\begin{align}\label{Asymptotic-Pi-Yu}
\bb P (|\Pi_{\upzeta_u} \mathcal{Y}_u| > u )   
  = \varkappa_s    \frac{ \kappa(s)^{-\chi_{\beta, l} (u) }  }{ \sqrt{\log u} }
    u^{-I(\beta - l)} ( 1 + o(1)). 
\end{align}
It remains to show that the first term $Y_{\upzeta_u}$ in \eqref{Pf_Yn_Equ01} negligible 
compared with the second term $\Pi_{\upzeta_u} \mathcal{Y}_u$. 
By Lemma \ref{Lem_tau_Low}, there exists a function $a(u) \downarrow 0$ ($u \to \infty$) 
such that uniformly in $l \in [0, (\log u)^{-\epsilon}]$, 
\begin{align}\label{Asypmtotic-Y-upzeta-u}
\bb P ( |Y_{\upzeta_u}| > a(u) u ) = o \left( \frac{u^{-I(\beta - l)}}{ \sqrt{ \log u } } \right). 
\end{align}
In view of \eqref{Pf_Yn_Equ01}, we have $Y_{\beta_u} = Y_{\upzeta_u} + \Pi_{\upzeta_u} \mathcal{Y}_u$. 
Hence
\begin{align*}
\bb P ( |Y_{\beta_u}| > u ) \geq \bb P ( |\Pi_{\upzeta_u} \mathcal{Y}_u| > u )
\end{align*}
and
\begin{align*}
\bb P ( |Y_{\beta_u}| > u ) 
& =  \bb P ( |Y_{\beta_u}| > u, |Y_{\upzeta_u}| \leq a(u) u )
   + \bb P ( |Y_{\beta_u}| > u, |Y_{\upzeta_u}| > a(u) u )  \nonumber\\
& \leq \bb P ( |\Pi_{\upzeta_u} \mathcal{Y}_u| > (1 - a(u))u ) + \bb P ( |Y_{\upzeta_u}| > a(u) u ). 
\end{align*}
Therefore, using \eqref{Asymptotic-Pi-Yu} and \eqref{Asypmtotic-Y-upzeta-u},
we conclude the proof of \eqref{Petrov_LD_tau}. 
The proof of \eqref{Petrov_LD_tau_y} is similar 
by using \eqref{Pf_Thm1_Decom_P-y}, \eqref{Pf_Thm1_Yu_Au-y} and \eqref{Pf_Thm1_Yu_Bu}. 
\end{proof}

\subsection{Proof of Theorem \ref{Cor_LLT_LD_tau}}
We only give a proof of the first assertion in Theorem \ref{Cor_LLT_LD_tau} since the second one can be carried out in the same way. 
By Theorem \ref{Thm_LD_Perpe_Petrov}, for any fixed $a \in \bb R_-$ and $m \in \bb Z_+$ with $a + m \leq 0$, 
there exists a sequence $(r_u)_{u >0}$ determined by the
matrix law $\mu$ satisfying $r_u \to 0$ as $u \to \infty$, such that  
uniformly in  $l \in [0, (\log u)^{-\epsilon}]$, 
\begin{align}\label{Pf_Pet_LD_tau}
 \bb P (\tau_u \leq (\beta - l) \log u + a + m)  
& = \frac{\varkappa_s \kappa(s)^{- (\beta - l) \log u - (a + m) + \floor{(\beta - l) \log u + (a + m)} } }{\sqrt{\log u}} 
   u^{-I \big( \beta - l + \frac{a + m}{\log u} \big)} (1 + r_u)  \nonumber \\
& = \frac{\varkappa_s \kappa(s)^{- (\beta - l) \log u - a + \floor{(\beta - l) \log u + a} } }{\sqrt{\log u}} 
   u^{-I \big( \beta - l + \frac{a + m}{\log u} \big)} (1 + r_u), 
\end{align}
where in the last equality we used the fact that $m \in \bb Z$. 
We make the difference of \eqref{Pf_Pet_LD_tau} with $m = 0$ and with $m > 0$ to obtain that,
as $u \to \infty$,
\begin{align}\label{Pf_LLT_LD_Diff}
& \bb P \Big(\tau_u - (\beta - l) \log u  \in (a, a + m ] \Big)  
 = \frac{\varkappa_s \kappa(s)^{- (\beta - l) \log u - a + \floor{(\beta - l) \log u + a} } }{\sqrt{\log u}}  \nonumber\\
& \qquad \qquad \qquad \times    
 \left[ u^{-I \big( \beta - l + \frac{a + m}{\log u} \big)}  (1 + r_u) 
      -  u^{-I \big( \beta - l + \frac{a}{\log u} \big)}  (1 + r_u')  \right]. 
\end{align}
From the definition of the rate function $I$, we see that it is analytic. 
By Lemma \ref{Lem_Expan_Rate} and Taylor's formula, we get that uniformly in $l \in [0, (\log u)^{-\epsilon}]$, 
\begin{align*}
u^{-I \big( \beta - l + \frac{a + m}{\log u} \big)} 
& = u^{-I( \beta - l)} u^{- I'(\beta - l) \frac{a + m}{\log u}}  (1 + o(1))  \nonumber\\ 
& = u^{-I( \beta - l)} e^{- (a + m) I'(\beta - l) }  (1 + o(1))  \nonumber\\ 
& = u^{-I( \beta - l)} e^{- (a + m) I'(\beta) }  (1 + o(1)),   
\end{align*}
and 
\begin{align*}
u^{-I \big( \beta - l + \frac{a}{\log u} \big)}
 = u^{-I( \beta - l)} u^{- I'(\beta - l) \frac{a}{\log u}}  (1 + o(1)) 
 =  u^{-I( \beta - l)} e^{- a I'(\beta) }  (1 + o(1)).  
\end{align*}
Hence 
\begin{align*}
  u^{-I \big( \beta - l + \frac{a + m}{\log u} \big)}  (1 + r_u) 
      -  u^{-I \big( \beta - l + \frac{a}{\log u} \big)}  (1 + r_u')  
 =  u^{-I( \beta - l)} e^{- a I'(\beta) } \left[ e^{- m I'(\beta) } - 1 \right]  (1 + o(1)).  
\end{align*}
Substituting this into \eqref{Pf_LLT_LD_Diff} finishes the proof of Theorem \ref{Cor_LLT_LD_tau}. 

\begin{acks}[Acknowledgments]
The authors would like to thank the referees  and the editor for 
very constructive comments and pointing out additional references which contributed to improve the presentation of the paper. 

Hui Xiao is corresponding author. 
\end{acks}
\begin{funding}
%
The authors  were supported by DFG grant ME 4473/2-1. 
Hui Xiao was also supported by the National Natural Science Foundation of China (No. 12288201). 
\end{funding}


\end{document}